\documentclass{amsart}
\usepackage[english]{babel}
\usepackage{amsmath}
\usepackage{amssymb}
\usepackage{amsthm}
\usepackage{mathrsfs}
\usepackage{hyperref}
\usepackage{physics}
\usepackage[pdftex]{graphicx}

\newtheorem{thm}{Theorem}[section]
\newtheorem{lem}[thm]{Lemma}
\newtheorem{prop}[thm]{Proposition}
\newtheorem{defn}[thm]{Definition}

\newtheorem{coro}[thm]{Corollary}

\newtheorem{ques}{Question}[section]

\newtheorem{rmk}[thm]{Remark}

\DeclareMathOperator{\diverge}{div}

\DeclareMathOperator{\dist}{dist}
\DeclareMathOperator{\Lip}{Lip}

\DeclareMathOperator{\diam}{diam}
\DeclareMathOperator{\supp}{supp}

\DeclareMathOperator{\Span}{span}
\DeclareMathOperator{\vol}{Vol}

\title[Partial Regularity for Minimal Surface System]{Partial regularity for Lipschitz solutions to the minimal surface system}
\author{Bryan Dimler}
\address{Department of Mathematics, UC Irvine}
\email{\tt bdimler@uci.edu}

\begin{document}

\maketitle

\begin{abstract}
In this paper, we study the regularity of several notions of Lipschitz
solutions to the minimal surface system with an emphasis on partial
regularity results. These include stationary solutions, integral weak
solutions, and viscosity solutions. We also prove an interior gradient
estimate for classical solutions to the system using the maximum
principle, assuming the area-decreasing condition and that all but one
component have small $L^\infty$ norm.
\end{abstract}

\section{Introduction}

If $\Gamma \subset \mathbb{R}^{n+m}$ is an $n$-dimensional Lipschitz submanifold, $\Omega \subset \mathbb{R}^n$ is a domain, and $F: \Omega \rightarrow \Gamma$ is a global parameterization of $\Gamma$\footnote{That is, $F$ is a bi-Lipschitz map with $\Gamma = F(\Omega)$.}, then the area of $\Gamma$ is given by the \emph{area formula}
\begin{equation}
    \mathcal{A}(\Gamma) := \int_\Omega \sqrt{\det(DF^TDF)} \, dx. \label{area1}
\end{equation}
The submanifold $\Gamma$ is said to be a \emph{minimal submanifold of} $\mathbb{R}^{n+m}$ if $\Gamma$  is a critical point of the area functional \eqref{area1} with respect to compact $C^2$ variations of $\mathbb{R}^{n+m}$. Minimal submanifolds are often called \emph{stationary submanifolds of} $\mathbb{R}^{n+m}$. By the first variation formula \eqref{1stvar}, $\Gamma$ is minimal if and only if $F$ satisfies the \emph{minimal surface system} in $\Omega$ in the weak sense for any local parameterization of $\Gamma$:

\begin{equation}
    \Delta_g F^\beta := \frac{1}{\sqrt{g}}\pdv{x^i}(\sqrt{g}g^{ij} \pdv{F^\beta}{x^j}) = 0 \text{ for each } \beta = 1,\ldots, n + m, \label{FMSS}
\end{equation}
where $g := (g_{ij})$ is the Riemannian metric on $\Gamma$ given locally in $F$-coordinates by
\begin{equation}
    g_{ij} := \pdv{F}{x^i} \cdot \pdv{F}{x^j}. \label{Fmetric}
\end{equation}
and $\sqrt{g} := \sqrt{\det g}$. 

When $\Gamma := \mathcal{G}_u$ is the graph of a Lipschitz function $u: \Omega \rightarrow \mathbb{R}^m$ and $F(x) = (x, u(x))$, the area formula \eqref{area1} becomes
\begin{equation}
    \mathcal{A}(\mathcal{G}_u) := \int_\Omega \sqrt{\det(I + Du^TDu)} \, dx, \label{area2}
\end{equation}
and the system \eqref{FMSS} reduces to
\begin{equation}
    \begin{cases}
        \sum_{i = 1}^n \pdv{x^i}(\sqrt{g}g^{ij}) = 0, \text{ } j = 1,\ldots, n \\
        \sum_{i,j = 1}^n \pdv{x^i}(\sqrt{g}g^{ij} \pdv{u^\beta}{x^j}) = 0, \text{ } \beta = 1,\ldots, m, \label{MSS}
    \end{cases}
\end{equation}
where $g$ is the usual metric on the graph of $u$ defined component-wise by
\begin{equation}
    g_{ij} = \delta_{ij} + \sum_{\beta = 1}^m \pdv{u^\beta}{x^i}\pdv{u^\beta}{x^j}. \label{metric}
\end{equation}
Lipschitz stationary submanifolds are invariant under rigid motions and solutions to the system \eqref{MSS} are invariant under rigid motions of their graphs. Furthermore, when $u \in C^{2}$ the system \eqref{MSS} simplifies to
\begin{equation}
    g^{ij}u_{ij}^\beta = 0 \text{ for each } \beta = 1, \ldots,m.
\end{equation}
This reflects the fact that the coordinate functions $x^i$ for each $i = 1,\ldots,n$ are harmonic with respect to the Laplace-Beltrami operator $\Delta_g$ on $\mathcal{G}_u$ when $\mathcal{G}_u$ is a $C^2$ minimal submanifold. When $m =1$, the minimal surface system reduces to the \emph{minimal surface equation}:
\begin{equation}
    \diverge_{\mathbb{R}^n}\Bigg(\frac{\nabla u}{\sqrt{1 + |\nabla u|^2}}\Bigg) = 0 \text{ in } \Omega. \label{MSE}
\end{equation}

Minimal submanifolds have been widely studied in mathematics and have many applications to the physical sciences. For example, techniques developed to study minimal submanifolds have been influential in the regularity theory of nonlinear elliptic partial differential equations, as well as in geometric measure theory and the calculus of variations. Applications to the physical sciences include modeling of soap films and soap bubbles, which is related to the isoperimetric inequality, as well as free boundary problems, phase transitions, crystallization, and general relativity.

In the present paper, we study the regularity of various notions of Lipschitz solutions to the system \eqref{MSS} with an emphasis placed on partial regularity results. We begin by showing that the singular set of a Lipschitz stationary solution to the system \eqref{MSS} defined on a domain in $\mathbb{R}^n$ has Hausdorff dimension at most $n-4$. We then show that this result can be improved to $n-5$ in the case of special Lagrangian graphs. Next, we study the regularity of a new notion of weak solution to the minimal surface system based on the maximum principle called a \emph{viscosity solution}\footnote{See Definition \ref{viscdef} below and the original work \cite{Sa2}.}. Specifically, we identify conditions under which one can deduce Lipschitz viscosity solutions are stationary, as well as conditions under which one can get full regularity for Lipschitz viscosity solutions from their components. We also show that the singular set of a Lipschitz viscosity solution on a domain in $\mathbb{R}^n$ has Hausdorff dimension at most $n-\epsilon$ for some $\epsilon > 0$ depending on $n$ and the Lipschitz constant. After studying the regularity for Lipschitz viscosity solutions, we briefly show that Lipschitz integral weak solutions which are invariant under small coordinate rotations are stationary. This result serves as a partial resolution to Conjecture 2.1 in \cite{LO}. To conclude, we prove an interior gradient estimate for smooth solutions to the system \eqref{MSS} using the maximum principle assuming the \emph{area-decreasing condition} (see Definition \ref{AD}) and that all but one component is small in the $L^\infty$ norm. The proof is in the spirit of Korevaar's proof in \cite{Ko} for constant mean curvature equations. Before continuing to the main results of this paper, we discuss some important background that will be serve as a base for many of our results.

Most of the work on minimal submanifolds has focused on the codimension one case. This is, in large part, due to the counter-examples of Lawson and Osserman in \cite{LO} in higher codimension which show a very different situation with respect to the minimal surface equation. For example, when $n \geq 4$ we are not guaranteed classical solvability of the Dirichlet problem for the minimal surface system on the unit ball $B_1^n(0) \subset \mathbb{R}^n$, even with smooth boundary data. In fact, classical solvability is open when $n = 3$. However, when $n = 2$ we are guaranteed classical solvability of the Dirichlet problem for continuous boundary data and every $m \geq 1$ thanks to the early work of Douglas and Rado (see \cite{Do}, \cite{Ra}, or \cite{LO}). 

In addition to non-existence, Lawson and Osserman showed that uniqueness of classical solutions to the Dirichlet problem on $B_1^n(0)$ fails, including when $n = m = 2$. They also showed that the map $u: \mathbb{R}^4 \rightarrow \mathbb{R}^3$ defined by 
\begin{equation}
    u(x) := \frac{\sqrt{5}}{2}|x| \eta\Big(\frac{x}{|x|}\Big), \label{LOcone}
\end{equation}
where $\eta : \mathbb{R}^4 \rightarrow \mathbb{R}^3$ is the \emph{Hopf map} $\eta(z_1,z_2) := (|z_1|^2 - |z_2|^2, z_1\overline{z_2})$, is a Lipschitz stationary solution to the minimal surface system. Since $u$ has a singularity at the origin, we have irregularity when $n= 4$.

Though many classic results for the codimension one case fail in higher codimension, there are still some positive results. For example, global Lipschitz stationary solutions are linear when $n \leq 3$ and, as a consequence, Lipschitz stationary solutions are smooth when $n \leq 3$ (see \cite{Fi}). Moreover, the solvability of the Dirichlet problem and other Bernstein theorems hold under specific bounds involving the principle values of $Du$ (see \cite{Fi}, \cite{JX}, and \cite{MW2}). Nonetheless, due to the lack of results in high codimension and the discrepancy between the codimension one and high codimension case, it is important to further study the regularity and existence theory for the minimal surface system.

\section{Preliminaries}

In this section, we state some key definitions and discuss several convergence results which will be used frequently in this paper.

\subsection{Hausdorff Measure}
For $A \subset \mathbb{R}^{n+m}$, $0 \leq k < \infty$, and $0 < \delta \leq \infty$, we define the outer measure $\mathcal{H}_\delta^k$ on $\mathbb{R}^{n+m}$ by
$$
    \mathcal{H}_\delta^k(A) = \omega_k 2^{-k} \inf \Big\{\sum_{j = 1}^\infty (\diam U_j)^k : A \subset \bigcup_{j = 1}^\infty U_j, \text{ } \diam U_j < \delta \Big\},
$$
where the infimum is taken over all open covers $\{U_j\}_{j \in \mathbb{N}}$ of A in $\mathbb{R}^{n+m}$ with the $U_j$ having diameter less than $\delta$ and where $\omega_k$ is the measure of the unit ball in $\mathbb{R}^k$. Then the \emph{$k$-dimensional Hausdorff measure} in $\mathbb{R}^{n+m}$ is defined on the Borel sets by 
$$
    \mathcal{H}^k(A) := \lim_{\delta \downarrow 0} \mathcal{H}_\delta^k(A) = \sup_{\delta >0} \mathcal{H}_\delta^k(A).
$$
The \emph{Hausdorff dimension} of a Borel measurable set $A \subset \mathbb{R}^{n+m}$ is defined to be 
$$
    \mathcal{H}_{\text{dim}}(A) := \inf\{k \geq 0 : \mathcal{H}^k(A) = 0\}.
$$

\subsection{Notions of Solution}
The first class of solutions we will study are \emph{stationary solutions}. Let $\Gamma$ be an $n$-dimensional Lipschitz submanifold of $\mathbb{R}^{n+m}$. We define a \emph{variation of $\Gamma$} to be a family of diffeomorphisms 
$$
    \varphi_t : \mathbb{R}^{n+m} \rightarrow \mathbb{R}^{n+m} \text{ where } t \in (-1,1)
$$
such that 
\begin{enumerate}
    \item $\varphi_t(x) \in C^2(\mathbb{R}^{n+m} \times (-1,1))$
    \item There is a compact set $K$ with $K \cap \partial \Gamma = \emptyset$ such that $\varphi_t(x) = x$ for each $x \notin K$ and each $t \in (-1,1)$.
    \item $\varphi_0(x) = x$ for each $x \in \mathbb{R}^{n+m}$.
\end{enumerate}
Define $\Gamma_t := \varphi_t(\Gamma)$ and let 
\begin{equation}\label{X}
    X := \pdv{\varphi_t}{t} \Big \lvert_{t = 0} \in C_0^1(\mathbb{R}^{n+m}; \mathbb{R}^{n+m}). 
\end{equation}
The \emph{support of a variation} $\varphi_t$ is defined to be the support of the vector field $X$. 

\begin{defn}[Stationary Submanifold]\label{stationary}
    Let $\Gamma$ be an $n$-dimensional Lipschitz submanifold of $\mathbb{R}^{n+m}$ and let $U \subset \mathbb{R}^{n+m}$ be an open set with $U \cap \Gamma \neq \emptyset$. We say $\Gamma$ is stationary in $U$ if 
    $$
        \dv{t} \Big\lvert_{t = 0} \mathcal{A}(\Gamma_t) = 0
    $$
    with respect to all variations $\varphi_t$ having compact support in $U$ vanishing in a neighborhood of $\partial \Gamma \cap U$. When $\Gamma$ is the graph of a Lipschitz function $u: \Omega \rightarrow \mathbb{R}^m$, we say that $u$ is stationary.
\end{defn}

Observe that if $F: \Omega \rightarrow \Gamma$ is a global parameterization of $\Gamma$, then by the \emph{first variation formula}
\begin{equation}
    \dv{t} \Big\lvert_{t = 0} \mathcal{A}(\Gamma_t) = - \int_{\Gamma} \Delta_g F \cdot X \, d\mathcal{H}^n = \int_{\Omega} g^{ij} \pdv{F^\beta}{x^j} \pdv{(X^\beta \circ F)}{x^i} \sqrt{g}\, dx \label{1stvar}
\end{equation}
where $X = X^i \pdv{F}{x^i}$. Equivalently, 
\begin{equation}
    \dv{t} \Big\lvert_{t = 0} \mathcal{A}(\Gamma_t) = \int_\Gamma \diverge_{g} X \, d \mathcal{H}^n
\end{equation}
where $\diverge_g X$ is the divergence of X along $\Gamma$:
\begin{equation}
    \diverge_g X := \frac{1}{\sqrt{g}} \pdv{x^i}(\sqrt{g}X^i).
\end{equation}
Furthermore, if $\Gamma$ is a $C^2$ submanifold then we may take $F \in C^2$ and $\Delta_g F = H_g$ where $H_g$ is the mean curvature vector for $\Gamma$. Thus, in the $C^2$ case stationarity is equivalent to the vanishing of the mean curvature vector. It is not hard to see that $C^2$ classical solutions of \eqref{MSS} are stationary and, vice versa, that if $u \in C^2$ is stationary $u$ solves \eqref{MSS} in the classical sense. 

We will also be interested in \emph{integral weak solutions}, which we will refer to as weak solutions for brevity. Since Lipschitz functions defined on a domain $\Omega$ are differentiable $\mathcal{H}^n$-a.e. by Rademacher's Theorem, it makes sense to define a weak solution as follows:

\begin{defn}[Weak Solution]\label{wksoln}
A Lipschitz function $u: \Omega \rightarrow \mathbb{R}^m$ is a weak solution to the minimal surface system provided
\begin{equation}
    \int_{\Omega} g^{ij} \pdv{u^\beta}{x^j} \pdv{\varphi^\beta}{x^i} \sqrt{g} \, dx= 0 \text{ for all } \varphi \in C_0^1(\Omega; \mathbb{R}^m), \label{weak}
\end{equation}
where $g$ is the usual Riemannian metric on the graph of $u$ given by \eqref{metric}.
\end{defn}
By taking $X(x,z) = (0,\ldots, 0, \varphi^1(x), \ldots, \varphi^m(x))$ in Definition \ref{stationary}, we see that Lipschitz stationary solutions are weak solutions. In addition, one can show that Lipschitz weak solutions are stationary when $m=1$. However, it is an open problem to determine whether Lipschitz weak solutions are necessarily stationary when $m \geq 2$ (see Conjecture 2.1 in \cite{LO}). In Section 5, we provide a partial solution to this problem. More specifically, we show that Lipschitz weak solutions which are invariant under small coordinate rotations are stationary. 

Since the maximum principle does not extend to graphs when the codimension is higher than one, only the variational approach has been successful in the general theory of nonlinear elliptic systems. However, in \cite{Sa2} Savin defined a notion of viscosity solution for the minimal surface system as follows:

\begin{defn}[Viscosity Solution]\label{viscdef}
    \begin{enumerate}
        \item We say that a $C^2$ function $G(x,z) : \mathbb{R}^{n+m} \rightarrow \mathbb{R}$ is a comparison function for the minimal surface system in the open set $U \subset \mathbb{R}^{n+m}$ if at any point $(x,z) \in U$ we have 
        $$
            \Delta_L G(x,z) > 0
        $$
        for any $n$-dimensional vector space $L$ which is normal to $\nabla G(x,z)$. Here, $\nabla G(x,z)$ is the gradient of $G$ in $\mathbb{R}^{n+m}$ at $(x,z)$ and $\Delta_L G(x,z)$ is the Laplacian of the restriction of $G$ to the vector space $L$ passing through $(x,z)$.
        \item Let $u : \Omega \rightarrow \mathbb{R}^m$ be a continuous function. We say that $u$ is a viscosity solution to the minimal surface system \eqref{MSS} if its graph cannot touch the level set of a comparison function $\{G = c\}$ from the side $\{G \leq c\}$.
    \end{enumerate}
\end{defn}

The following proposition is immediate from the definitions.

\begin{prop}\label{C2}
    A function $u \in C^2(\Omega; \mathbb{R}^m)$ is a viscosity solution to the minimal surface system if and only if $u$ is a $C^2$ classical solution of \eqref{MSS}.
\end{prop}

The definition of viscosity solution can be extended to $n$-dimensional\footnote{In the sense of Hausdorff dimension.} compact sets in $\mathbb{R}^{n+m}$ instead of just graphs if an additional density bound is assumed. In this setting, the viscosity solution definition is weaker than the one provided by the theory of varifolds, since the support of a stationary varifold cannot be tangent to $\{G = c\}$ from the side $\{G \leq c\}$ with $G$ a comparison function. This is due to the fact that the projection onto the level set $\{G \leq c- \epsilon\}$ from the side $\{G \geq c-\epsilon\}$ decreases the $n$-dimensional area. Remarkably, Savin proved an Allard Theorem for viscosity solutions requiring only $L^\infty$ closeness locally to an affine function (see \cite{Sa2}). The Allard Theorem will be at the core of many of our arguments.

\subsection{Convergence Results}
The results in this section will be used frequently throughout this paper. The first of these is a compactness theorem for viscosity solutions to the system.
\begin{prop}[Stability]\label{compact}
    If $\{u_j\}$ is a sequence of viscosity solutions to the minimal surface system and $u_j \rightarrow u$ uniformly on compact sets, then $u$ is a viscosity solution also.
\end{prop}

One of the most crucial tools in this paper is Savin's Allard Theorem (Theorem 1.5 in \cite{Sa2}).
\begin{thm}[Allard Theorem]\label{Allard}
    Assume $u : B_1^n(0) \rightarrow \mathbb{R}^m$ is a viscosity solution to the minimal surface system such that 
    $$\norm{u - \ell}_{L^\infty(B_1^n(0))} \leq \epsilon,$$
    where $\ell(x) = b + Ax$ is an affine function $\mathbb{R}^n \rightarrow \mathbb{R}^m$. Then there is an $\epsilon_0 := \epsilon_0(n,m,|A|) > 0$ such that $u \in C^2(B_{2^{-1}}^n(0); \mathbb{R}^m)$ and
    $$\norm{D^2u}_{C^\alpha(B_\frac{1}{2}^n(0))} \leq C\epsilon$$
    if $\epsilon \leq \epsilon_0$.
\end{thm}
\begin{rmk}
    By the linear theory for elliptic systems, we can conclude even more; namely, that viscosity solutions satisfying the hypotheses of the Allard Theorem are real analytic in $B_{2^{-1}}^n(0)$.
\end{rmk}

Expanding as a Taylor series at each point of differentiability, interior regularity for viscosity solutions to \eqref{MSS} is immediate:
\begin{coro}[Interior Regularity for Viscosity Solutions]\label{intreg}
    Suppose $u: \Omega \rightarrow \mathbb{R}^m$ is a differentiable viscosity solution to the minimal surface system on the domain $\Omega$. Then $u$ is real analytic.
\end{coro}

Note that the condition that $u$ is differentiable cannot be reduced to the requirement that $u$ is only Lipschitz due to the Lawson-Osserman example \eqref{LOcone}. As an easy application of the Allard Theorem and stability, we see that the singular set of a viscosity solution is closed.
\begin{prop}\label{conv1}
    Suppose $\{u_j\}$, $u_j : \Omega \rightarrow \mathbb{R}^m$, is a sequence of viscosity solutions to the minimal surface system converging locally uniformly to a viscosity solution $u: \Omega \rightarrow \mathbb{R}^m$ of the minimal surface system. Let $\Sigma$ denote the set of singular points for $u$ in $\Omega$ and let $x_0 \in \Omega \setminus \Sigma$. Let $\{x_j\}$ be a sequence of points in $\Omega$ such that $x_j \rightarrow x_0$. Then for sufficiently large $j$, $x_j$ is a regular point of $u_j$ (i.e. $x_j \in \Omega \setminus \Sigma$). 
\end{prop}
\begin{proof}
    Since $x_0$ is a regular point for $u$, expanding in a Taylor series near $x_0$ shows $u(x) = \ell(x) + O(|x - x_0|^2)$ for some affine function $\ell$ (i.e. $\ell$ is the linear part of $u$ at $x_0$). Let $\epsilon_0$ be as in Theorem \ref{Allard}. Choosing $0 < r < \epsilon_0$ small enough, we find 
    $$
    \norm{u - \ell}_{L^\infty(B_{r}^n(x_0))} < \frac{\epsilon_0}{2}r.
    $$ 
    Now, $u_j \rightarrow u$ locally uniformly on $\Omega$ so we may choose $N \in \mathbb{N}$ so large that 
    $$
        \norm{u_j - u}_{L^\infty(B_{r}^n(x_0))} < \frac{\epsilon_0}{2}r \text{ for each } j \geq N.
    $$
    Then 
    $$
    \norm{u_j - \ell}_{L^\infty(B_{r}^n(x_0))} < \epsilon_0 r \text{ for each } j \geq N.
    $$
    By Theorem \ref{Allard}, $u_j$ is regular in a neighborhood of $x_0$ for $j \geq N$. Since $x_j \rightarrow x_0$, we conclude that $x_j$ is a regular point for $u_j$ for large enough $j$.
\end{proof}
\begin{coro}\label{conv2}
  Suppose $\{u_j\}$ is a sequence of viscosity solutions $u_j : \Omega \rightarrow \mathbb{R}^m$ to the minimal surface system. Suppose $u_j \rightarrow u$ locally uniformly on $\Omega$ for some viscosity solution $u :\Omega \rightarrow \mathbb{R}^m$ to the minimal surface system. Let $\Sigma$ be the singular set for $u$ and for each $j \in \mathbb{N}$ let $\Sigma_j$ be the singular set for $u_j$. Suppose $K \subset \Omega$ is compact and that $\Sigma \cap K \subset U$ for some open set $U \subset \Omega$. Then for large enough $j$ we have $\Sigma_j \cap K \subset U$.
\end{coro}
\begin{proof}
    Suppose the conclusion is false. Passing possibly to a subsequence, we may suppose that for each $j \in \mathbb{N}$ there is an $x_j \in \Sigma_j \cap K$ with $x_j \notin U$. Since $K$ is compact, we may also assume $x_j \rightarrow \overline{x} \in K$. By Proposition \ref{conv1}, we must have $\overline{x} \in \Sigma$ for otherwise $x_j$ is a regular point for $u_j$ for $j$ large enough. Hence, $\overline{x} \in \Sigma \cap K \subset U$. Since $U$ is open and $x_j \notin U$ for each $j$, this contradicts the assumption $x_j \rightarrow \overline{x}$. Thus, the lemma holds.
\end{proof}

The following three lemmata concern the $k$-dimensional Hausdorff measure on $\mathbb{R}^n$ and will be helpful in proving partial regularity for Lipschitz stationary solutions and Lipschitz viscosity solutions. The first two are stated without proof. For the proofs of Lemma \ref{null} and Lemma \ref{dense}, we refer the reader to Chapter 11 in \cite{G}.

\begin{lem}\label{null}
    For every $A \subset \mathbb{R}^{n+m}$, $\mathcal{H}_\infty^k(A) = 0$ if and only if $\mathcal{H}^k(A) = 0$.
\end{lem}

\begin{lem}[Uniform Density Estimate]\label{dense}
    If $A \subset \mathbb{R}^{n+m}$ such that $\mathcal{H}_\infty^k(A) > 0$, then for $\mathcal{H}^k$-a.e. $x \in A$ we have
$$
    \limsup_{r\downarrow 0} \frac{\mathcal{H}_\infty^k(A \cap B_r(x))}{\omega_k r^k} \geq 2^{-k}.
$$
\end{lem}
\begin{lem}[Lower Semi-continuity] \label{uscontinuity}
    Let $\{u_j\}_{j \in \mathbb{N}}$, $u$, $\Sigma_j$, and $\Sigma$ be as in Corollary \ref{conv2}. Then for any compact subset $K \subset \Omega$ we have
$$
    \mathcal{H}_\infty^k(\Sigma \cap K) \geq \limsup_j \mathcal{H}_\infty^k(\Sigma_j \cap K).
$$
\end{lem}
\begin{proof}
    Let $\epsilon > 0$ be given and let $\{ U_j\}_{j \in \mathbb{N}}$ be an open covering of $\Sigma \cap K$ such that 
$$
    \mathcal{H}_\infty^k(\Sigma \cap K) > \omega_k2^{-k}\sum_j(\diam U_j)^k - \epsilon.
$$
Let $U := \bigcup_j U_j$ so that $U$ is open and $\Sigma \cap K \subset U$. By Corollary \ref{conv2}, for large enough $j$ we find $\Sigma_j \cap K \subset U$. Hence,
$$
    \mathcal{H}_\infty^k(\Sigma_j \cap K) \leq \mathcal{H}_\infty^k(U) \leq \omega_k2^{-k} \sum_{l}(\diam U_l)^k 
$$
so 
$$
    \limsup_j \mathcal{H}_\infty^k(\Sigma_j \cap K) \leq \mathcal{H}_\infty^k(\Sigma \cap K) + \epsilon.
$$
Since $\epsilon >0$ is arbitrary, the proof is complete.
\end{proof}

\section{Partial Regularity for Lipschitz Stationary Solutions}
It is well known that Lipschitz stationary solutions are smooth when $n \leq 3$ (see \cite{Fi}, for example). We will show that the Hausdorff measure of the singular set of a Lipschitz stationary solution in a domain $B_1^n(0)$ is at most $n - 4$. After doing so, we prove a similar result for Lipschitz special Lagrangian graphs. Of course, these results extend to general domains $\Omega \subset \mathbb{R}^{n}$ by a covering argument. In \cite{NV2} and \cite{WY}, non-Lipschitz singular solutions to the minimal surface system were constructed in dimension $n = 3$. In fact, the examples constructed are special Lagrangian and volume-minimizing. These results demonstrate that the hypothesis that the solution $u$ is Lipschitz is critical to our results. 

The following proposition is a consequence of the remarks for stationary varifolds following Proposition \ref{C2}. It shows that the results from the previous section can be applied to Lipschitz stationary solutions of the minimal surface system.
\begin{prop}\label{statvisc}
    If $u: \Omega \rightarrow \mathbb{R}^m$ is a Lipschitz stationary solution to the minimal surface system in a domain $\Omega$, then $u$ is a viscosity solution.
\end{prop}
\begin{rmk}\label{rmk1}
    Since every Lipschitz function on a domain $\Omega \subset \mathbb{R}^n$ is differentiable $\mathcal{H}^n$-a.e. in $\Omega$ by Rademacher's Theorem, expanding in a Taylor series near each point of differentiability and using the Allard Theorem once more shows that a Lipschitz viscosity solution is analytic away from a set of $\mathcal{H}^n$-measure zero. 
\end{rmk}

Though all Lipschitz stationary solutions are viscosity solutions, it is not known whether the converse holds in general. However, when $m = 1$ Lipschitz viscosity solutions are stationary. Later, we discuss various sufficient conditions for a Lipschitz viscosity solution to be stationary.

As in the codimension one case, partial regularity for the minimal surface system is proved by a dimension reduction procedure via blow-up of Lipschitz stationary solutions near singular points. Let $u : \mathbb{R}^n \rightarrow \mathbb{R}^m$ be a Lipschitz stationary solution to the minimal surface system. We say that $u$ is a \emph{cone in $\mathbb{R}^n$ with vertex at the origin} if $u$ is first-order homogeneous; that is, $u(rx) = ru(x)$ for each $x \in \mathbb{R}^n$. Notice that, in this case, the graph map $F(x) = (x, u(x))$ is first-order homogeneous as well so that the graph of $u$ is an $n$-dimensional cone in $\mathbb{R}^{n+m}$ with vertex at the origin. 

For fixed $x_0 \in \Omega$ and $\lambda > 0$ small, define the function $u_{\lambda}: B_{\lambda^{-1}}^n(0) \rightarrow \mathbb{R}^m$ by 
$$
u_{\lambda}(x) := \frac{1}{\lambda}(u(x_0 + \lambda x) - u(x_0)).
$$
Propositions \ref{bup} and \ref{bupc} below are a consequence of Allard's Theorem and the monotonicity formula for Lipschitz stationary submanifolds of $\mathbb{R}^{n+m}$.

\begin{prop}[Blow-up]\label{bup}
    Let $u: \Omega \rightarrow \mathbb{R}^m$ be a Lipschitz stationary solution to the minimal surface system. Then there is a sequence $\lambda(i) \rightarrow 0$ such that $u_{\lambda(i)} \rightarrow v$ locally uniformly in $\mathbb{R}^n$, where $v: \mathbb{R}^n \rightarrow \mathbb{R}^m$ is a Lipschitz stationary solution to the minimal surface system and its graph, $\mathcal{G}_v$, is a minimal cone with vertex at the origin.
\end{prop}
\begin{prop}[Blow-up of a Cone]\label{bupc}
    Let $u: \mathbb{R}^n \rightarrow \mathbb{R}^m$ be a Lipschitz stationary solution to the minimal surface system whose graph is a cone with vertex at the origin. Let $x_0 \in \mathbb{R}^n\setminus \{0\}$ and let $u_{\lambda}$ be as above. Then there is a sequence $\lambda(i) \rightarrow 0$ such that $u_{\lambda(i)} \rightarrow v$ locally uniformly in $\mathbb{R}^n$ to a Lipschitz stationary solution to the minimal surface system. Moreover, $\mathcal{G}_v$ is a product of the form $C \times \mathbb{R}$ (up to rotation), where $C$ is a minimal cone of dimension $n-1$ in $\mathbb{R}^{n + m -1}$ which is also a graph. 
\end{prop}
The proofs for Proposition \ref{bup} and Proposition \ref{bupc} can be found in chapter 11 of \cite{MG}.

\newpage
\begin{prop}\label{mincone}
    Suppose $u: B_1^n(0) \rightarrow \mathbb{R}^m$ is a Lipschitz stationary solution to the minimal surface system in $B_1^n(0)$ and that $\mathcal{H}^k(\Sigma) >0$, where $\Sigma$ is the singular set for $u$. Then there exists a Lipschitz stationary solution $v: \mathbb{R}^n \rightarrow \mathbb{R}^m$ such that $\mathcal{G}_v$ is a minimal cone and $\mathcal{H}^k(\Sigma_v) > 0$, where $\Sigma_v$ denotes the singular set of $v$.
\end{prop}
\begin{proof}
    By Lemma \ref{dense}, there is an $x_0 \in \Sigma$ and a sequence $\{r_j\}_{j \in \mathbb{N}}$ such that $r_j \rightarrow 0$ and 
    $$\mathcal{H}_\infty^k(\Sigma \cap B_{r_j}^n(x_0)) \geq 2^{-k-1}\omega_kr_j^k.$$
    Without loss of generality, we may assume $x_0 = 0$. For each $\lambda > 0$, let $u_{\lambda}$ be as above with $x_0 = 0$. By Proposition \ref{bup}, there is a subsequence $\{r_{n_j}\}$ such that $u_j := u_{r_{n_j}} \rightarrow v$ locally uniformly in $B_1^n(0)$, where $v : \mathbb{R}^n \rightarrow \mathbb{R}^m$ is a Lipschitz solution to the minimal surface system such that $\mathcal{G}_v$ is a minimal cone with vertex at the origin. For each $j$, let $\Sigma_j$ be the singular set for $u_j$. It is clear that $\Sigma_j := \{ x \in \mathbb{R}^n : r_{n_j}x \in \Sigma\}$. Hence,
$$
    \mathcal{H}_\infty^k(\Sigma_j \cap B_1^n(0)) = r_{n_j}^{-k}\mathcal{H}_\infty^k(\Sigma \cap B_{r_{n_j}}^n(0)) \geq 2^{-k-1}\omega_k.
$$
The result now follows from Lemma \ref{null} and Lemma \ref{uscontinuity}.
\end{proof}

Using Proposition \ref{mincone}, we may now make progress toward the proof of partial regularity in the case that $u$ is a Lipschitz stationary solution. We first prove that the singular set of a Lipschitz stationary solution to the minimal surface system consists of at most isolated points when $n = 4$. This theorem will serve as the base case in the proof of the main result, whose proof is similar to the codimension one case (see \cite{G} Chapter 11). 

\begin{prop}\label{isolated}
    If $u: B_1^4(0) \rightarrow \mathbb{R}^m$ is a Lipschitz stationary solution of the minimal surface system, then the singular set $\Sigma$ consists of at most isolated points.
\end{prop}
\begin{proof}
    Suppose there is a sequence of singular points $\{x_n\}$ for $u$ converging to some $x_0 \in B_1^n(0)$. Without loss of generality, we may assume $x_0 = 0$. Set $r_n := |x_n|$ for each $n \in \mathbb{N}$ (so that $r_n \rightarrow 0$). Let $u_\lambda$ be as in the proof of Proposition \ref{mincone}. By Proposition \ref{bup}, there is a subsequence $\{r_{n_j}\}$ such that $u_j := u_{r_{n_j}}$ such that $u_j \rightarrow v$ locally uniformly, where $v : \mathbb{R}^4 \rightarrow \mathbb{R}^m$ is a Lipschitz stationary solution of the minimal surface system whose graph $\mathcal{G}_v$ is a minimal cone with vertex at the origin. Set $y_n := \frac{x_n}{r_n}$ for each $n \in \mathbb{N}$. Since $\partial B_1^n(0)$ is compact and $y_n \in \partial B_1^n(0)$ for each $n \in \mathbb{N}$, we may assume $y_n \rightarrow y_0$ for some $y_0 \in \partial B_1^n(0)$. By Proposition \ref{conv1}, $y_0$ is a singular point for $v$. Since $\mathcal{G}_v$ is a graphical cone with vertex at the origin (i.e. $v(rx) = rv(x)$ for each $r \geq 0$), the whole line joining $0$ to $y_0$ consists of singular points for $v$. Let $\Sigma_v$ denote the set of singular points for $v$. Then $\mathcal{H}^1(\Sigma_v) >0$ by the discussion above. 
    
    Next, we blow up the cone $v$ near $y_0$ as in Proposition \ref{bupc} to obtain a Lipschitz stationary solution $w :\mathbb{R}^4 \rightarrow \mathbb{R}^m$ to the minimal surface system whose graph is of the form $\mathcal{G}_w = \mathcal{G}_{\tilde{w}} \times \mathbb{R}$ (up to rotation), where $\tilde{w} : \mathbb{R}^3 \rightarrow \mathbb{R}^m$ is a Lipschitz stationary solution to the minimal surface system whose graph is a $3$-dimensional cone in $\mathbb{R}^{3 + m}$. That is, there is an orthonormal system of coordinates $x = (x^1, x^2, x^3, x^4)$ in $\mathbb{R}^4$ and a function $\tilde{w}: \mathbb{R}^3 \rightarrow \mathbb{R}^m$ and some $\sigma \in \mathbb{R}^m$ such that $w(x) = \sigma x^4 + \tilde{w}(\tilde{x})$ where $\tilde{x} := (x^1, x^2, x^3)$ in $\mathbb{R}^3$ and $\mathcal{G}_{\tilde{w}}$ is a $3$-dimensional Lipschitz minimal cone. Arguing as in Proposition \ref{mincone}, we see that $\mathcal{H}^1(\Sigma_w) > 0$. Furthermore, the expression for $w$ in terms of $\tilde{w}$ shows that if $x$ is a singular point for $w$ then $\tilde{x}$ is a singular point for $\tilde{w}$. Likewise, if $\tilde{x}$ is a singular point for $\tilde{w}$, then $(\tilde{x}, x^4)$ is a singular point for $w$ for each $x^4 \in \mathbb{R}$. It follows that $\Sigma_w = \Sigma_{\tilde{w}} \times \mathbb{R}$. Since $\mathcal{H}^1(\Sigma_w) > 0$, we must have $\mathcal{H}^0(\Sigma_{\tilde{w}}) > 0$ for otherwise $\Sigma_w = \Sigma_{\tilde{w}} \times \mathbb{R} = \emptyset$. However, this contradicts the non-existence of singular Lipschitz stationary cones in $\mathbb{R}^3$ (see \cite{Fi}). 
\end{proof}

We may now prove the main theorem for this section.

\begin{thm}[Partial Regularity for Lipschitz Stationary Solutions]\label{partreg}
    If $u: B_1^n(0) \rightarrow \mathbb{R}^{m}$ is a Lipschitz stationary solution of the minimal surface system, then $u$ is smooth away from a closed singular set of Hausdorff dimension at most $n - 4$. 
\end{thm}
\begin{proof}
    Let $k > 0$ be such that $\mathcal{H}^k(\Sigma) > 0$. By Proposition \ref{mincone}, there is a Lipschitz stationary solution $v: \mathbb{R}^n \rightarrow \mathbb{R}^m$ to the minimal surface system such that its graph $\mathcal{G}_v$ is a minimal cone and $\mathcal{H}^k(\Sigma_v) >0$. By Lemma \ref{dense}, we can find $x_0 \neq 0$ and a sequence $r_j \rightarrow 0$ such that 
$$
    \mathcal{H}_\infty^k(\Sigma_v \cap B_{r_j}^n(x_0)) \geq 2^{-k-1}\omega_kr_j^k.
$$
Blowing up $\mathcal{G}_v$ near $x_0$, we obtain a Lipschitz stationary solution $w: \mathbb{R}^n \rightarrow \mathbb{R}^m$ to the minimal surface system whose graph is a minimal cylinder of the form $\mathcal{G}_w= \mathcal{G}_{\tilde{w}} \times \mathbb{R}$, where $\tilde{w}: \mathbb{R}^{n-1} \rightarrow \mathbb{R}^m$ is a Lipschitz stationary solution to the minimal surface system and $\mathcal{G}_{\tilde{w}}$ is a minimal cone. In particular, $\mathcal{H}^{k-1}(\Sigma_{\tilde{w}}) >0$, for otherwise $\mathcal{H}^k(\Sigma_w) = 0$ since in this case the Hausdorff dimension of $\Sigma_w$ is at most $k-1$. Repeating this argument, we conclude that for each $l \leq k$ there is a Lipschitz stationary solution to the minimal surface system in $\mathbb{R}^{n-l}$, $u_l$, such that $\mathcal{G}_{u_l}$ is a minimal cone in $\mathbb{R}^{n-l+m}$ with $\mathcal{H}^{k-l}(\Sigma_{u_l}) > 0$. Combining this with Proposition \ref{isolated} implies $k \leq n-4$. That $\Sigma$ is closed is an immediate consequence of Proposition \ref{conv1}.
\end{proof}

\begin{rmk}
    By the example \eqref{LOcone} provided by Lawson and Osserman in \cite{LO}, the $n-4$ in Theorem \ref{partreg} is optimal.
\end{rmk}

We close this section with some remarks about the partial regularity of stationary Lipschitz Lagrangian submanifolds. A \emph{Lagrangian submanifold} $\Gamma$ of $\mathbb{R}^{2n}$ is a submanifold that can be represented as the gradient graph of a scalar function over a domain $\Omega \subset \mathbb{R}^n$. That is, 
$$
    \Gamma := \{(x,\nabla u(x)) : x \in \Omega, \text{ } u: \Omega \rightarrow \mathbb{R}\}.
$$
We call $u$ the \emph{potential function for} $\Gamma$. It is well known that if $\Gamma$ is a stationary Lipschitz Lagrangian submanifold with potential function $u$, then $u$ solves the \emph{special Lagrangian equation} for a.e. $x \in \Omega$:
\begin{equation}
    \arctan D^2u = \sum_{i = 1}^n \arctan \lambda_i = C, \label{lagrangian}
\end{equation}
where the $\lambda_i$ are the (real) eigenvalues of the symmetric matrix $D^2u$ and $C$ is a constant. When $\nabla u$ is Lipschitz, the equation \eqref{lagrangian} is uniformly elliptic. Likewise, if $u$ solves \eqref{lagrangian} for a.e. $x \in \Omega$ and $\nabla u$ is Lipschitz, then $\nabla u$ is a Lipschitz stationary solution to the minimal surface system.

In \cite{NV}, it was shown that if $u$ is a two homogeneous real analytic function in $\mathbb{R}^4 \setminus \{0\}$ which solves a fully nonlinear uniformly elliptic equation of the form $F(D^2u) = 0$ in $\mathbb{R}^4 \setminus\{0\}$, then $u$ is a quadratic polynomial; hence, smooth. It is not hard to show that if $\nabla u$ is one homogeneous, then $u$ is two homogeneous. Thus, if $\nabla u$ is a Lipschitz one homogeneous stationary solution to \eqref{MSS}, then $u$ is a two homogeneous solution to the fully nonlinear uniformly elliptic equation \eqref{lagrangian}. We show that such a solution $\nabla u$ to the system \eqref{MSS} has singular set with Hausdorff dimension at most $n-5$. Therefore, in the case of Lipschitz stationary Lagrangian graphs, one can improve the $n-4$ in Theorem \ref{partreg} to $n-5$.

\begin{thm}[Partial Regularity for Special Lagrangian Graphs]\label{pregSLAG}
    Suppose $u: B_1^n(0) \rightarrow \mathbb{R}^n$ is a Lipschitz stationary solution to the minimal surface system. If $\mathcal{G}_u$ is a Lagrangian graph, then $u$ is smooth away from a closed singular set of Hausdorff dimension at most $n-5$.
\end{thm}

Since the proof of the theorem is nearly identical to that of Theorem \ref{partreg}, we provide only a sketch of the proof.
\begin{proof}[Proof Sketch]
    We first show that the singular set $\Sigma$ for $u$ contains at most isolated points when $n = 5$. Let $w: B_1^n(0) \rightarrow \mathbb{R}$ be the potential function for $u$; i.e. $u = \nabla w$. Suppose $\mathcal{H}^1(\Sigma) > 0$. As in the proof of Proposition \ref{isolated}, assume there is a sequence of singular points for $u$ such that $x_j \rightarrow x_0 =0$. Repeating the previous argument, we obtain a sequence $u_j := u_{r_j}$ ($r_j \rightarrow 0$) of Lipschitz stationary solutions with $\Lip u_j = \Lip u =: M$ converging locally uniformly to a function $\overline{u} : \mathbb{R}^5 \rightarrow \mathbb{R}^5$ which is a Lipschitz stationary solution to the system whose graph is a minimal cone with vertex at the origin. Furthermore, $\mathcal{H}^1(\Sigma_{\overline{u}}) > 0$ where $\Sigma_{\overline{u}}$ is the singular set for $\overline{u}$. Notice that for each $j$, the graph of $u_j$ is Lagrangian with $u_j = \nabla w_j$ where 
    \begin{align*}
        w_j(x) &:= r_j^{-2}w(r_j x) \\
        \nabla w_j(x) &:= r_j^{-1}\nabla w(r_j x)
    \end{align*}
    solve the special Lagrangian equation \eqref{lagrangian} and the minimal surface system \eqref{MSS} for each $j$, respectively. It follows that the graph of $\overline{u}$ is Lagrangian with $\overline{u} = \nabla v$ where
    \begin{equation*}
        v = \lim_j w_j \text{ and } \nabla v = \lim_j \nabla w_j
   \end{equation*}
   and $v$ solves the special Lagrangian equation \eqref{lagrangian} with the same constant. In summary, we have obtained a one homogeneous Lipschitz stationary solution $\overline{u}$ to the minimal surface system whose graph is Lagrangian in $\mathbb{R}^{10}$ with two homogeneous potential function $v$ on $\mathbb{R}^5$ solving the special Lagrangian equation with the same constant. In particular, $\mathcal{H}^1(\Sigma_{\overline{u}}) > 0$. Since $\Sigma_{\overline{u}}$ contains no $C^1$ points by the Allard Theorem, $\mathcal{H}^1(\Sigma_v) > 0$ also. Arguing precisely as in Proposition \ref{isolated}, we blow up $v$ by quadratic re-scalings \begin{align*}
        v_j(x) := s_j^2v(s_jx) 
   \end{align*}
   near a singular point $y_0$ for $v$ (hence, also for $\overline{u}$) to obtain a two homogeneous function $\overline{v}: \mathbb{R}^5 \rightarrow \mathbb{R}$ solving the special Lagrangian equation which, after an appropriate choice of coordinates, can be written in the form
   $$
    \overline{v}(x) := cx_n^2 + \tilde{v}(\tilde{x})
   $$
   where $c \in \mathbb{R}$ is a fixed constant and $\tilde{v}: \mathbb{R}^4 \rightarrow \mathbb{R}$ solves the special Lagrangian equation \eqref{lagrangian} in $\mathbb{R}^4$. Likewise, we obtain a Lipschitz stationary solution to the minimal surface system $w$ with potential function $\overline{v}$ whose graph is a five dimensional cone in $\mathbb{R}^{10}$ with singularity at the origin. Using the expression for $\overline{v}$, we see that 
   $$
        w(x) = (0,0,0,0,cx^5) + (\nabla_{\mathbb{R}^4}\tilde{v}(\tilde{x}), 0).
   $$
   Set 
   $$
        \tilde{w}(\tilde{x}) := (\nabla_{\mathbb{R}^4}\tilde{v}(\tilde{x}), 0).
    $$
    Then $\tilde{w}$ has singular set with non-zero $\mathcal{H}^0$-measure so that $\nabla_{\mathbb{R}^4}\tilde{v}$ has singular set with non-zero $\mathcal{H}^0$-measure also. Since the singular set for $\nabla_{\mathbb{R}^4}\tilde{v}$ is a subset of that for $\tilde{v}$, we deduce that $\tilde{v}$ is a two homogeneous solution to the special Lagrangian equation with non-trivial singular set, which is impossible. The case of general $n > 5$ now follows by an inductive procedure as in the proof of Theorem \ref{partreg}. Closure of the singular set is immmediate since $u$ is stationary.
\end{proof}

It is important to note that it is not known whether there exist non-flat graphical special Lagrangian cones. Hence, we cannot claim that the $n-5$ in Theorem \ref{pregSLAG} is optimal. If $k$ is the smallest dimension in which one exists ($k \geq 5$ by \cite{NV}), then the $n-5$ in Proposition \ref{pregSLAG} can be improved to $n-k$.
 
\section{Regularity of Lipschitz Viscosity Solutions}
Now, we turn our attention to the regularity theory for Lipschitz viscosity solutions. Using the theory of varifolds, we prove a necessary and sufficient condition for a Lipschitz viscosity solution to be stationary, along with other conditions that imply stationarity. In fact, these results will follow from more general facts concerning stationary extensions of Lipschitz submanifolds which are smooth with vanishing mean curvature off of a small set in measure. After doing so, we prove that the singular set of a Lipschitz viscosity solution defined on a domain in $\mathbb{R}^n$ has Hausdorff dimension at most $n-\epsilon$ for $\epsilon > 0$ depending on the Lipschitz constant and $n$. We conclude this section by demonstrating that maximum principle techniques can be used to deduce regularity of Lipschitz viscosity solutions to the system when certain conditions are imposed on their component functions (e.g. all but one component is $C^{1,\alpha}$).

\subsection{Stationary Lipschitz Viscosity Solutions}
Recall that a \emph{countably \\$n$-rectifiable set in $\mathbb{R}^{n+m}$} is a set $\Gamma \subset \mathbb{R}^{n+m}$ such that $\Gamma \subset N \cup (\cup_{i = 1}^\infty F_j(A_j))$ where $\mathcal{H}^n(N) = 0$ and $F_j: A_j \subset \mathbb{R}^n \rightarrow \mathbb{R}^{n+m}$ are Lipschitz functions. We assume in addition that $\mathcal{H}^n(K \cap \Gamma) < \infty$ for any compact set $K \subset \mathbb{R}^{n+m}$. Given a pair $(\Gamma, \theta)$ where $\Gamma \subset \mathbb{R}^{n+m}$ is a countably $n$-rectifiable $\mathcal{H}^n$-measurable set and a non-negative function $\theta \in L_{\text{loc}}^1(\mathbb{R}^{n+m}; \mathcal{H}^n)$, we define the \emph{$n$-rectifiable varifold} $V:=V(\Gamma, \theta)$ to be the equivalence class of all pairs $(\tilde{\Gamma}, \tilde{\theta})$ where $\tilde{\Gamma}$ is countably $n$-rectifiable, $\theta = \tilde{\theta}$ $\mathcal{H}^n$-a.e. on $\Gamma \cap \tilde{\Gamma}$, and $\mathcal{H}^n(\Gamma \triangle \tilde{\Gamma}) = 0$\footnote{Here, $A \triangle B$ denotes the symmetric difference $(A\setminus B) \cup (B \setminus A)$.}. We define the \emph{first variation of $V$} to be 
$$
     \delta V(X) = \int_{\Gamma} \diverge_\Gamma X \, d \mu_V \text{ for } X \in C_0^1(\mathbb{R}^{n+m}; \mathbb{R}^{n+m}).
$$
where $\mu_V$ is the \emph{weight measure}
$$
    \mu_V(A) := \int_{A \cap \Gamma} \theta \, d\mathcal{H}^n.
$$
It is common to assume $\theta \equiv 0$ on $\mathbb{R}^{n+m} \setminus \Gamma$. Hence, $\Gamma$ is the support of the measure $\mu_V$. The function $\theta$ is called the \emph{multiplicity function}. A varifold $V$ is said to be \emph{stationary} if $\delta V \equiv 0$. 

Note that if $\Gamma \subset \mathbb{R}^{n+m}$ is an $n$-dimensional Lipschitz submanifold, then $\Gamma$ can be represented as the varifold $V(\Gamma, 1)$. Moreover, if $\varphi_t$ is any variation of $\Gamma$ with initial velocity $X \in C_0^1(\mathbb{R}^{n+m}; \mathbb{R}^{n+m})$ vanishing in a neighborhood of $\partial \Gamma$, the first variation formula for stationary Lipschitz submanifolds gives 
\begin{equation}
    \delta \Gamma(X) = \dv{t} \Big\lvert_{t = 0} \mathcal{A}(\varphi_t(\Gamma)). \label{delta}
\end{equation}
Consequently, when $V$ is a Lipschitz submanifold, stationarity in the sense of varifolds is equivalent to stationarity of Lipschitz submanifolds. For a thorough treatment of the theory of varifolds, see \cite{Si}.

We will make use of the following definitions:

\begin{defn}\label{genmeancurv}
    \begin{enumerate}
        \item We say a rectifiable $n$-varifold $V(\Gamma, \theta)$ in an open set $U \subset \mathbb{R}^{n+m}$ has generalized mean curvature vector $H_\Gamma$ if there is a Borel function $H_\Gamma: U \rightarrow \mathbb{R}^{n+m}$ satisfying $H_\Gamma \in L_{\text{loc}}^1(U, \mathbb{R}^{n+m}; \mu_V)$ and
    $$
        \int_{\Gamma} \diverge_{\Gamma} X \, d\mu_V =-\int_{\Gamma} X\cdot H_\Gamma \, d\mu_V \text{ for any } X \in C_0^1(U; \mathbb{R}^{n+m})
    $$
    \item A rectifiable $n$-varifold $V(\Gamma, \theta)$ in an open set $U \subset \mathbb{R}^{n+m}$ has locally bounded first variation if for each compact set $K \subset \subset U$ there is a constant $C_K \geq 0$ depending on $K$ such that for each $X \in C_0^1(U; \mathbb{R}^{n+m})$ with $\supp X \subset K$,
    $$
        |\delta V(X)| \leq C_K \sup_K |X|.
    $$
    \end{enumerate}
\end{defn}

When $\Gamma$ is a Lipschitz submanifold of $\mathbb{R}^{n+m}$, we denote the Levi-Civita connection, divergence along $\Gamma$, and mean curvature vector for $\Gamma$ by $\nabla_g$, $\diverge_g$ and $H_g$, respectively. In addition, when $\Gamma$ is the graph of a Lipschitz function $u: \Omega \rightarrow \mathbb{R}^m$ on a domain $\Omega$, we say $u$ has mean curvature vector $H_g$ on $\Omega$. For the remainder of this section, $\Omega$ will always denote a general domain in $\mathbb{R}^n$ and $V:= V(\Gamma, \theta)$ will denote an $n$-rectifiable varifold in $\mathbb{R}^{n+m}$ with support $\Gamma$ and multiplicity function $\theta$.

By the Allard Theorem (specifically, see Remark \ref{rmk1}), a Lipschitz viscosity solution to the system \eqref{MSS} is a smooth classical solution to the minimal surface system outside of a closed set $\Sigma$ with $\mathcal{H}^n(\Sigma) = 0$. It thereby makes sense to study when Lipschitz submanifolds of $\mathbb{R}^{n+m}$ whose graphs are smooth minimal submanifolds outside of a small set in measure can be extended to be stationary submanifolds. 

Notice that, given a rectifiable $n$-varifold $V$, if the generalized mean curvature vector exists and is zero $\mu_V$-a.e., then $V$ is stationary. Suppose now that $V = V(\Gamma, 1)$ where $\Gamma$ is an $n$-dimensional Lipschitz submanifold of $\mathbb{R}^{n+m}$ and let $U \subset \mathbb{R}^{n+m}$ be an open set with $\Gamma \cap U \neq \emptyset$. If there is a closed set $\Sigma \subset \Gamma \cap U$ (i.e. in the subspace topology on $\Gamma$) such that $\mathcal{H}^{n}(\Sigma) = 0$ and $(\Gamma \cap U) \setminus \Sigma$ is a smooth minimal submanifold of $\mathbb{R}^{n+m}$, then for $x_0 \in (\Gamma \cap U) \setminus \Sigma$ a cut-off argument in a local coordinate chart centered at $x_0$ shows that if the generalized mean curvature for $\Gamma$ exists it must be zero. We record this observation:

\begin{prop}\label{existH}
    Suppose $\Gamma$ is an $n$-dimensional Lipschitz submanifold of $\mathbb{R}^{n+m}$ and that $U \subset \mathbb{R}^{n+m}$ is open with $U \cap \Gamma \neq \emptyset$. Assume $\Gamma \cap U$ is a smooth submanifold of $\mathbb{R}^{n+m}$ with vanishing mean curvature outside of a closed set $\Sigma \subset \Gamma \cap U$ with $\mathcal{H}^n(\Sigma) = 0$. If the generalized mean curvature for $\Gamma$ exists in $U$, then $\Gamma$ is stationary in $U$.
\end{prop}
By Proposition \ref{existH}, it is worthwhile to examine when the generalized mean curvature vector exists for a given $n$-rectifiable varifold $V$ in an open set $U \subset \mathbb{R}^{n+m}$. Notice that we may view $\delta V$ as a linear functional on $C_0^1(U; \mathbb{R}^{n+m})$. Hence, if $V$ has locally bounded first variation the linear functional $\delta V$ can be extended to a continuous linear functional on $C_0(U; \mathbb{R}^{n+m})$. By the Riesz Representation Theorem, there is a vector-valued Radon measure $\mu_\delta$ on $U$ such that 
$$
    \delta V(X) = \int_{\Gamma} X \cdot d\mu_\delta \text{ for every } X \in C_0(U; \mathbb{R}^{n+m}). 
$$
Using the Radon-Nikodym Theorem, we can derive $\mu_\delta$ with respect to the weight measure $\mu_V$ to obtain a function $H \in L_{\text{loc}}^1(U, \mathbb{R}^{n+m}; \mu_V)$ with values in $\mathbb{R}^{n+m}$ and a measure $\mu_s$ singular to $\mu_V$ such that 
$$
    d\mu_\delta = H \, d\mu_V+ d\mu_s.
$$
Comparing with the definitions above, we see that $V$ has generalized mean curvature vector $H_\Gamma = -H$ if and only if $d\mu_\delta$ is absolutely continuous with respect to $d\mu_V$. This is equivalent to the condition
$$
    |\delta V(X)| \leq C\int_{\Gamma} |X|f \, d\mu_V \text{ for each } X \in C_0^1(U; \mathbb{R}^{n+m})
$$
for some non-negative $f \in L_{\text{loc}}^1(U; \mu_V)$ and some fixed constant $C \geq 0$. We thereby have a characterization: 

\begin{thm}\label{character1}
    Suppose $\Gamma$ is an $n$-dimensional Lipschitz submanifold in $\mathbb{R}^{n+m}$ and let $U \subset \mathbb{R}^{n+m}$ be an open set with $\Gamma \cap U \neq \emptyset$. If $\Gamma \cap U$ is a smooth submanifold of $\mathbb{R}^{n+m}$ with vanishing mean curvature outside of a closed set $\Sigma \subset \Gamma \cap U$ with $\mathcal{H}^n(\Sigma) = 0$, then $\Gamma$ is stationary in $U$ if and only if 
    \begin{equation}
         |\delta \Gamma(X)| \leq C\int_{\Gamma} |X|f \, d\mathcal{H}^n \text{ for each } X \in C_0^1(U; \mathbb{R}^{n+m}), 
    \end{equation}
    where $f \in L_{\text{loc}}^1(U; \mathcal{H}^n)$ is non-negative and $C \geq 0$ is a fixed constant.
\end{thm}

As an application of Theorem \ref{character1}, we consider the case when $\Gamma := \mathcal{G}_u$ is the graph of a Lipschitz function $u: \Omega \rightarrow \mathbb{R}^{m}$. If $U \subset \mathbb{R}^{n+m}$ is an open set such that $\mathcal{G}_u \cap U \neq \emptyset$ and $X \in C_0^1(U; \mathbb{R}^{n+m})$ vanishes near $\partial \mathcal{G}_u$, $X(x,z) = (X^1(x,z), \ldots, X^{n+m}(x,z))$, we know that if $F(x) := (x,u(x))$ is the graph map 
$$
    \delta \mathcal{G}_u(X) = \int_{\Omega}  g^{ij} \pdv{F^\beta}{x^j} \pdv{(X^\beta \circ F)}{x^i} \sqrt{g} \, dx.
$$
If $u \in W_{\text{loc}}^{2,1}(\Omega, \mathbb{R}^{m}; \mathcal{H}^n)$, then we may define the weak Laplacian of $F$ component-wise by \eqref{FMSS}, which exists and is locally integrable for $u \in W_{\text{loc}}^{2,1}(\Omega, \mathbb{R}^{m}; \mathcal{H}^n)$ by the local integrability of $D^2 u$ and boundedness of $Du$ (hence, the metric and its inverse as well). Integrating by parts shows
$$
    \delta \mathcal{G}_u(X) = -\int_{\mathcal{G}_u} \Delta_g F \cdot X \, d\mathcal{H}^n.
$$
It follows that the generalized mean curvature $H_g$ exists and is equal to $\Delta_g F$ $\mathcal{H}^n$-a.e. on $\mathcal{G}_u$. Appealing to Proposition \ref{existH} or Theorem \ref{character1}, we conclude:

\begin{coro}\label{character3}
     Suppose $u: \Omega  \rightarrow \mathbb{R}^m$ is a Lipschitz function which is a smooth classical solution to the minimal surface system outside of a closed set $\Sigma \subset \Omega$ with $\mathcal{H}^n(\Sigma) = 0$. If $u \in W_{\text{loc}}^{2,1}(\Omega,\mathbb{R}^m; \mathcal{H}^n)$, then $u$ is stationary.
\end{coro}

Since Lipschitz viscosity solutions are smooth classical solutions to the minimal surface system outside of a closed $\mathcal{H}^n$-null set in their domain, the results above apply to Lipschitz viscosity solutions. For completeness, we state the theorems in this case.

\begin{thm}[Characterization of Stationarity]\label{cstat}
    Suppose $u:\Omega \rightarrow \mathbb{R}^m$ is a Lipschitz viscosity solution to the minimal surface system and let $U \subset \mathbb{R}^{n+m}$ be an open set with $\mathcal{G}_u \cap U \neq \emptyset$. Then $u$ is stationary in $U$ if and only if 
    \begin{equation}
         |\delta \mathcal{G}_u(X)| \leq C\int_{\mathcal{G}_u} |X|f \, d\mathcal{H}^n \text{ for each } X \in C_0^1(U; \mathbb{R}^{n+m}), \label{character2}
    \end{equation}
    where $f \in L_{\text{loc}}^1(U; \mathcal{H}^n)$ is non-negative and $C \geq 0$ is a fixed constant.
\end{thm}
\begin{coro}
     Suppose $u: \Omega  \rightarrow \mathbb{R}^m$ is a Lipschitz viscosity solution to the minimal surface system. If $u \in W_{\text{loc}}^{2,1}(\Omega,\mathbb{R}^m; \mathcal{H}^n$), then $u$ is stationary.
\end{coro}

Our next goal is to show that similar results to those above can be obtained when we assume that the set of singularities of an $n$-dimensional Lipschitz submanifold in $\mathbb{R}^{n+m}$ has $\mathcal{H}^{n-1}$-measure zero. Applying this to the case of a Lipschitz graph, we obtain a removal of singularities theorem for both Lipschitz classical solutions and Lipschitz viscosity solutions to the system \eqref{MSS}. 

In the theorem below, $\overline{\Sigma}$ denotes the closure of $\Sigma$ in $\mathbb{R}^{n+m}$.

\begin{prop}\label{LipStat}
   Suppose $\Gamma$ is an $n$-dimensional Lipschitz submanifold of $\mathbb{R}^{n+m}$ such that $\Gamma$ is a smooth submanifold with vanishing mean curvature away from a closed set $\Sigma \subset \Gamma$ such that $\mathcal{H}^{n-1}(K \cap \overline{\Sigma}) = 0$ whenever $K \subset \mathbb{R}^{n+m}$ is compact and $K \cap \partial \Gamma = \emptyset$. Then $\Gamma$ is stationary.
\end{prop}
\begin{proof}
   Let $X \in C_0^1(\mathbb{R}^{n+m}; \mathbb{R}^{n+m})$ and suppose $K := \supp X$ satisfies $K \cap \partial \Gamma = \emptyset$. Set $K_\Sigma := K \cap \overline{\Sigma}$ and note that $K_\Sigma$ is compact with $\mathcal{H}^{n-1}(K_\Sigma) = 0$ by assumption. Thus, for each $k \in \mathbb{N}$ we can choose $r_k > 0$ so small that there is a covering $\{B_{k,l}\}_{l = 1}^{j_k}$ of $K_\Sigma$ by balls in $\mathbb{R}^{n+m}$ with radii equal to $r_k$ for each $l$ such that 
   $$
        \omega_{n-1}\sum_{l = 1}^{j_k} r_{k}^{n-1} < \frac{1}{4^{n}k}.
   $$
   Set $\Sigma_k := \bigcup_{l = 1}^{j_k} 4\overline{B}_{k,l}$ where $4\overline{B}_{k,l}$ is the closed ball with the same center as $\overline{B}_{k,l}$ but with radius $4r_k$. For each $k$, define
   $$
        D_k := \{ (x,z) \in \mathbb{R}^{n+m} : \dist((x,z), K_\Sigma) < r_{k}\}.
   $$
   If $(x,z) \in D_k$, then the ball of radius $r_k$ centered at $(x,z)$ intersects $K_\Sigma$ at a point $(x_0, z_0)$. Choose $l$ so that $(x_0,z_0) \in B_{k,l}$ and let $(x_1, z_1)$ be the center of $B_{k,l}$. Then
   $$
        |(x,z) - (x_1,z_1)| \leq 2r_{k}
   $$
   by the triangle inequality so $(x,z) \in 4\overline{B}_{k,l}$ for this $l$. Since this can be done for each $(x,z) \in D_k$, we know $D_k \subset \Sigma_k$ for each $k$. Moreover, $K_\Sigma \subset D_k$ for each $k$ also. 
   
   Now, for each $k$ we can choose $\varphi_k \in C^\infty(\mathbb{R}^{n+m})$ such that $0 \leq \varphi \leq 1$, $\varphi \equiv 1$ outside $D_k$, $\varphi_k = 0$ in a neighborhood of $K_\Sigma$, and $\norm{\nabla \varphi}_{\infty} \leq C r_k^{-1}$ for some uniform bound $C$ in $k$. Notice that $\varphi X$ is a $C_0^1$ vector field with compact support $K_\varphi \subseteq K$ satisfying $K_\varphi \cap K_\Sigma = \emptyset$. Using that $\Gamma$ is a smooth Lipschitz minimal submanifold away from $\Sigma$ along with the $L^\infty$ bounds for $\nabla_g \varphi_k$ and $X^T$, we find
   \begin{align*}
       \Big|\int_\Gamma \varphi_k \diverge_g X \, d\mathcal{H}^n\Big| &= \Big|\int_\Gamma \nabla_g \varphi_k \cdot X^T \, d \mathcal{H}^n \Big| \\
       &\leq C_1r_k^{-1}\sum_{l = 1}^{j_k}\mathcal{A}(\Gamma \cap 4\overline{B}_{k,l})\\
       &\leq C_2r_k^{-1}\sum_{l = 1}^{j_k} (4r_k)^{n} \\
       &\leq \frac{C_3}{k} \rightarrow 0 \text{ as } k \rightarrow \infty.
   \end{align*}
   By the dominated convergence theorem, we have
   $$
        0 = \lim_{k \rightarrow \infty} \int_\Gamma \varphi_k \diverge_g X \, d\mathcal{H}^n = \int_\Gamma \diverge_g X \, d\mathcal{H}^n.
   $$
   Since $X$ is arbitrary, $\Gamma$ is stationary by the first variation formula.
\end{proof}

As a consequence of Proposition \ref{LipStat}, we get a removal of singularities theorem for Lipschitz classical solutions and Lipschitz viscosity solutions to the minimal surface system. To see this, let $\Gamma := \mathcal{G}_u$ be the graph of a Lipschitz function $u: \Omega \rightarrow \mathbb{R}^m$ and suppose $u$ is a smooth classical solution to the minimal surface system outside of a closed set $\Sigma \subset \Omega$ with $\mathcal{H}^{n-1}(\Sigma) = 0$. Let $F(x) := (x,u(x))$ be the graph map and let $K \subset \mathbb{R}^{n+m}$ be a compact set with $K \cap \partial \mathcal{G}_u = \emptyset$. Set $K_\Sigma := K \cap \overline{F(\Sigma)}$. By a convergence argument, it is easy to see that $K_\Sigma = K \cap F(\Sigma)$. Now, $F^{-1}(K_\Sigma) = F^{-1}(K) \cap \Sigma$ is compact in $\Omega$ with $\mathcal{H}^{n-1}(F^{-1}(K_\Sigma)) = 0$ by assumption. Using that $F$ is Lipschitz, we see that $\mathcal{H}^{n-1}(K_\Sigma) = 0$. Since $K$ is arbitrary, appealing to Proposition \ref{LipStat} yields:

\begin{coro}[Removal of Singularities I]\label{remsing}
     If $u: \Omega \rightarrow \mathbb{R}^m$ is a Lipschitz function which is a smooth classical solution to the minimal surface system outside of a closed set $\Sigma \subset \Omega$ satisfying $\mathcal{H}^{n-1}(\Sigma) = 0$, then $u$ is stationary and the Hausdorff dimension of its singular set is at most $n - 4$. 
\end{coro}
\begin{coro}[Removal of Singularities II]\label{ViscStat}
     If $u: \Omega \rightarrow \mathbb{R}^m$ is a Lipschitz viscosity solution to the minimal surface system with singular set $\Sigma$ satisfying $\mathcal{H}^{n-1}(\Sigma) = 0$, then $u$ is stationary and the Hausdorff dimension of its singular set is at most $n-4$.
\end{coro}

Due to the Lawson-Osserman example \eqref{LOcone}, we cannot claim that the function $u$ in Corollary \ref{remsing} is a smooth classical solution to the minimal surface system in all of $\Omega$. Even more, the radial solution $u: \mathbb{R}^2 \rightarrow \mathbb{R}$ to the Monge-Amp\'ere equation 
$$
    \det D^2 u = 1 + \pi\delta_0 \text{ in } \mathbb{R}^2,
$$
where $\delta_0$ is the Dirac $\delta$-function with mass at the origin, given by 
$$
    u(x) = \int_0^{|x|} \sqrt{1 + s^2} \, ds
$$
has locally bounded gradient which solves the minimal surface system away from a singularity at the origin. Hence, we are not even guaranteed a smooth extension when $u$ has bounded singularities in the case $n = m = 2$. With these examples in mind, Corollary \ref{remsing} can be viewed as the natural analogue in higher codimension of the removal of singularities theorem for the minimal surface equation (see Theorem 16.9 in \cite{G}) and is a consequence of a more general fact about Lipschitz submanifolds of Euclidean space. Notice also that, when $n\leq 3$, Corollary \ref{remsing} together with Theorem \ref{partreg} imply that in the Lipschitz graphical case we always have a smooth (analytic, in fact) extension. However, the hypothesis that $u$ is Lipschitz is necessary for a smooth extension in lower dimensions due to the examples in \cite{NV} and \cite{NV2} for $n = 3$ and the fact that $z^{-1}: \mathbb{C} \setminus \{0\} \rightarrow \mathbb{C}$ solves the system when viewed as a map $\mathbb{R}^2 \setminus \{0\} \rightarrow \mathbb{R}^2$.

In light of Theorem \ref{cstat} and Corollary \ref{ViscStat}, to prove stationarity of Lipschitz viscosity solutions it suffices to either prove the inequality \eqref{character2} holds for a general Lipschitz viscosity solution or that a Lipschitz viscosity solution is necessarily locally $W^{2,1}$ on its domain. This provides a possible route to proving stationarity of Lipschitz viscosity solutions that does not rely on dimension reduction arguments, which are based on the existence of a monotonicity formula. By Corollary \ref{remsing}, if one could show that the singular set of a Lipschitz viscosity solution necessarily has $\mathcal{H}^{n-1}$-measure zero, then one could also conclude Lipschitz viscosity solutions are stationary. However, this method faces immediate difficulty due to the lack of a monotonicity formula. On the other hand, it is not clear if we should even expect Lipschitz viscosity solutions to be stationary since we are not guaranteed existence and uniqueness of the Dirichlet problem and there is a lack of available methods for constructing non-trivial viscosity solutions. 

\subsection{Partial Regularity}
In this section, we prove a partial regularity theorem for Lipschitz viscosity solutions. Namely, we show that if $u: \Omega \subset \mathbb{R}^n \rightarrow \mathbb{R}^m$ is a Lipschitz viscosity solution to the system, then the Hausdorff dimension of its singular set is at most $n-\epsilon$ where $\epsilon$ depends on the dimension and $\Lip u$. To do so, we first show that if it is known that there is a $0 < p \leq n$ such that every Lipschitz viscosity solution on $B_1^n(0)$ has singular set with $\mathcal{H}^p$-measure zero, then the Hausdorff dimension of the singular set can be reduced by some $\epsilon > 0$ depending on the Lipschitz constant and $p$. The partial regularity will then follow easily from the Allard Theorem. After doing so, we discuss a possible route for obtaining an estimate on $\epsilon$ based on the techniques in \cite{AS}, as well as some consequences and related problems. 

\begin{prop}\label{partial1}
    Suppose $u: B_1^n(0) \rightarrow \mathbb{R}^m$ is a Lipschitz viscosity solution to the minimal surface system and let $\Sigma$ be the singular set for $u$. If there is a $p \in (0,n]$ such that for any $u$ as above we have $\mathcal{H}^p(\Sigma) = 0$, then there is an $\epsilon > 0$ depending on $\Lip u$ and $p$ such that $\Sigma$ has Hausdorff dimension at most $p - \epsilon$. 
\end{prop}
\begin{proof}
    Suppose the conclusion is not true. Then there are constants $M > 0$ and $p > 0$ and a sequence of Lipschitz viscosity solutions $u_j$ satisfying $\Lip u_j = M$ and $\mathcal{H}^{p_j}(\Sigma_j) > 0$ where $\Sigma_j$ is the singular set of $u_j$ for each $j$ and $p_j$ increases to $p$ as $j \rightarrow \infty$. Since Lipschitz viscosity solutions are invariant under translations and the rescaling $v(x) \mapsto \lambda^{-1}v(\lambda x)$, we may assume without loss of generality that $0 \in \Sigma_j$ is an $\mathcal{H}^{p_j}$ point of density for each $j$. Applying Lemma \ref{dense}, we may choose a sequence $r_j$ such that $r_j \rightarrow 0$ and 
    $$
        \frac{\mathcal{H}_\infty^{p_j}(\Sigma_j \cap B_{r_j}^n(0))}{\omega_{p_j}r_j^{p_j}} \geq 2^{-p_j} \geq 2^{-p}
    $$
    for each $j$.
    
    Set $v_j := r_j^{-1}(u_j(r_jx) - u_j(0))$ for each $j \in \mathbb{N}$. Then $v_j$ is a Lipschitz viscosity solution with Lipschitz constant $M$ for each $j$. By the Arzel\'a-Ascoli Theorem, we may assume $v_j \rightarrow u$ uniformly on compact subsets of $B_1^n(0)$. Moreover, $\Lip u = M$ and $u$ is a viscosity solution by stability of viscosity solutions. Let $\tilde{\Sigma}_j$ be the singular set for $v_j$ for each $j$. Then $\tilde{\Sigma}_j = \{x \in B_1^n(0): r_jx \in \Sigma_j\}$. We show that $\mathcal{H}^p(\Sigma \cap B_{2^{-1}}^n(0)) > 0$. Since $\mathcal{H}^p(\Sigma) = 0$ by assumption, we will have reached a contradiction. By Lemma \ref{null}, it suffices to show 
    $$
        \mathcal{H}_\infty^p(\Sigma \cap B_{\frac{1}{2}}^n(0)) > 0. 
    $$
    Let $\{U_l\}$ be any open cover of $\Sigma \cap B_{2^{-1}}^n(0)$ and set $U := \bigcup_l U_l$. By Corollary \ref{conv2}, for large enough $j$ we have $\tilde{\Sigma}_j \cap B_{2^{-1}}^n(0) \subset U$. Then
    $$
        \mathcal{H}_\infty^{p_j}(\tilde{\Sigma}_j \cap B_{\frac{1}{2}}^n(0)) \leq \mathcal{H}_\infty^{p_j}(U) \leq \omega_{p_j}2^{-p_j}\sum_{l = 1}^\infty(\diam U_l)^{p_j} \leq \omega_{p}\sum_{l = 1}^\infty(\diam U_l)^{p_j}.
    $$
    when $j$ is large. Since we intend to take the infimum over all open coverings and the unit ball is an open cover of $\Sigma \cap B_{2^{-1}}^n(0)$, we may assume the right-hand side is finite. Then there is a $K \in \mathbb{N}$ such that $\diam(U_l) < 1$ whenever $l \geq K$. Fix $k \in \mathbb{N}$ and note that for $j > k$ and $l \geq K$ we have 
    $$
        (\diam U_l)^{p_j} \leq \diam(U_l)^{p_k}
    $$
     Thus, by the dominated convergence theorem
    $$
        \lim_{j \rightarrow \infty} \sum_{l = 1}^\infty (\diam U_l)^{p_j} =  \sum_{l = 1}^\infty (\diam U_l)^{p}.
    $$
    On the other hand,
   $$
    \mathcal{H}_\infty^{p_j}(\tilde{\Sigma}_j \cap B_{\frac{1}{2}}^n(0)) \geq 2^{-p_j}  \frac{\mathcal{H}_\infty^{p_j}(\Sigma_j \cap B_{r_j}^n(0))}{\omega_{p_j}r_j^{p_j}} \geq \omega_{p_1}2^{-2p}
   $$
   for each $j$ so that 
   \begin{equation*}
       \omega_{p_1}2^{-2p} \leq \omega_p \sum_{l = 1}^\infty (\diam U_l)^p = 2^p (\omega_p2^{-p}) \sum_{l = 1}^\infty (\diam U_l)^p.
   \end{equation*}
   Hence,
   $$
    \omega_{p_1} 2^{-3p} \leq \omega_p 2^{-p} \sum_{l = 1}^\infty (\diam U_l)^p.
   $$
   Taking the infimum over all open covers $\{U_l\}$ shows $\mathcal{H}_{\infty}^{p}(\Sigma \cap B_{2^{-1}}^n(0)) > 0$, as desired.
\end{proof}

By the Allard Theorem, the singular set of a Lipschitz viscosity solution has $\mathcal{H}^n$-measure zero. Hence, Proposition \ref{partial1} and a covering argument give partial regularity:

\begin{thm}[Partial Regularity Theorem]\label{partial2}
     Suppose $u: \Omega \rightarrow \mathbb{R}^m$ is a Lipschitz viscosity solution to the minimal surface system. Then there is an $\epsilon > 0$ depending on $\Lip u$ and $n$ such that $u$ is smooth away from a closed singular set of Hausdorff dimension at most $n - \epsilon$. 
\end{thm}

The drawback to the approach used above is that it gives no quantitative information about $\epsilon$. However, if one were to show that the components of a viscosity solution to the system are themselves viscosity solutions to some uniformly elliptic scalar equations with ellipticity constants depending on the Lipschitz constant, it would follow that $u \in W^{2,\epsilon}$ for some $\epsilon > 0$ universal (i.e. depends only on the dimension, $\lambda$, and $\Lambda$). Then, using the Allard Theorem, we may argue as in \cite{AS} to once again prove Theorem \ref{partial2} in the case that the component functions of a Lipschitz viscosity solution to the system are viscosity solutions to scalar equations. The benefit of this approach is that it allows one to obtain a lower bound for $\epsilon$ in the theorem dependent only on the dimension and Lipschitz constant. We demonstrate this approach below. 

\subsection{Regularity from Components}
In this section, we will be concerned with Lipschitz viscosity solutions $u: B_1^n(0) \rightarrow \mathbb{R}^m$ to the system \eqref{MSS} whose components satisfy uniformly elliptic scalar equations of the form
   \begin{equation}
        a^{ij}(x)u_{ij}^\beta(x) = 0 \text{ in } B_1^n(0) \label{uniellip}
   \end{equation}
in the viscosity sense, where $(a_{ij}(x))$ is a symmetric matrix for each $x$ with eigenvalues in $[\lambda, \Lambda]$ depending on the Lipschitz constant. Recall that for a symmetric matrix $N$, the \emph{maximal Pucci operator} is defined by 
$$
    \mathcal{M}_{\lambda, \Lambda}^+(N) = \Lambda|N^+| - \lambda|N^-|.
$$
When $\lambda = 1$, we suppress the $\lambda$ and write $\mathcal{M}_{\Lambda}^+(N)$. The main result for this section is as follows:

\begin{thm}\label{preg2}
    Suppose $u: B_1^n(0) \rightarrow \mathbb{R}^m$ is a Lipschitz viscosity solution to the minimal surface system such that each component $u^\beta$ is a viscosity solution to uniformly elliptic equations of the form \eqref{uniellip}. Let $\Sigma \subset B_1^n(0)$ be the singular set for $u$. Then $\Sigma$ has Hausdorff dimension at most $n-\epsilon$, where $\epsilon \geq C(n)(\frac{\Lambda}{\lambda})^{1-n}$.
\end{thm}
\begin{rmk}\label{preg6}
We note that, since the minimal surface system is invariant under the transformation $u \mapsto -u$, the assumption \eqref{uniellip} is equivalent to the condition that each $u^\beta$ cannot be touched from above by $\varphi \in C^2$ satisfying
    \begin{equation}
        \mathcal{M}_{\lambda, \Lambda}^+(D^2\varphi) < 0 \label{preg1}
    \end{equation}
    for some $0 < \lambda \leq \Lambda$.
\end{rmk}

To prove Theorem \ref{preg2}, we will need $W^{2,\epsilon}$ regularity for viscosity solutions to scalar equations. For any $M > 0$, let 
\begin{equation}
    G_M := \{ x \in B_1(0) : \exists \text{ an affine function } l \text{ satisfying } \norm{u - l}_{L^\infty(B_r(x))} \leq Mr^2\}.  
\end{equation}
For the proof of Proposition \ref{preg3} below, see \cite{AS} or Chapter 7 in \cite{CC}.

\begin{prop}[$W^{2,\epsilon}$ Regularity]\label{preg3}
    Suppose $u: B_1(0) \rightarrow \mathbb{R}$, $u \in C(B_1)$, is a viscosity solution to some uniformly elliptic equation of the form \eqref{uniellip} and $\norm{u}_\infty \leq 1$. Then for any $M > 0$
    $$
        \mathcal{H}^n(B_1(0) \setminus G_M) \leq CM^{-\epsilon},
    $$
    where $\epsilon$ and $C$ are universal.
\end{prop}

If $u$ is a Lipschitz viscosity solution to the system satisfying the hypotheses of Theorem \ref{preg2}, then Proposition \ref{preg3} and the Minkowski inequality show that $u \in W^{2,\epsilon}(B_1(0))$. Furthermore, if $x \in G_M$ then $u$ is $C^{\infty}$ in $B_{\frac{r}{2}}(x)$ provided $Mr < \epsilon_0$ by the Allard Theorem.
\begin{proof}[Proof of Theorem \ref{preg2}]
    Suppose $u$ is a Lipschitz viscosity solution to the minimal surface system with components $u^\beta$ satisfying \eqref{preg1}. Since the minimal surface system is invariant under the Lipschitz transformation $u \mapsto t^{-1}u(tx)$, we may assume without loss of generality that $\norm{u}_{\infty} \leq 1$. A covering argument then gives the general result. Let $\tilde{\Sigma}$ be the singular set for $u$ in $B_{2^{-1}}^n(0)$. Since $\tilde{\Sigma}$ is a closed subset of the compact set $\Sigma$, $\tilde{\Sigma}$ is compact also. Fix $M > 0$ and let $\epsilon_0$ be as in the Allard Theorem. Choose $r > 0$ so that $Mr < \epsilon_0$. By the Vitali Covering Lemma, there is a finite collection $\{B_r^n(x_l)\}_{l = 1}^k$ of disjoint balls with centers $x_l \in \tilde{\Sigma}$ such that 
    $$
        \tilde{\Sigma} \subset \bigcup_{l = 1}^k B_{3r}^n(x_l).
    $$
    Moreover,
    $$
        y \notin G_M \text{ for each } y \in \bigcup_{l = 1}^k B_r^n(x_l),
    $$
    for otherwise $x_l$ is a regular point for some $l$ by the Allard Theorem. In particular, we have $B_r^n(x_l) \subset B_1^n(0) \setminus G_M$ for each $l = 1,\ldots, k$. Applying Proposition \ref{preg3}, we conclude that
    $$
        kr^n \leq Ck\vol(B_r^n(0))\leq CM^{-\epsilon} \leq Cr^{\epsilon}
    $$
    for universal constants $C, \epsilon$ provided $M^{-1} < r$ also. Hence,
    $$
        \sum_{l = 1}^k \big(\diam(B_{3r}(x_l))\big)^{n-\epsilon} \leq C.
    $$
    It follows that $\mathcal{H}^{n-\epsilon}(\tilde{\Sigma}) \leq C$ for universal constants $C,\epsilon$. Using translation invariance of viscosity solutions and a standard covering argument, we deduce that $\mathcal{H}^{n-\epsilon}(\Sigma) \leq C$. Consequently, $\mathcal{H}_{\dim}(\Sigma) \leq n - \epsilon$. The estimate for $\epsilon$ follows from Theorem 2.5 in \cite{Mo}.
\end{proof}

\begin{rmk}
    Notice that if it can be shown that $\epsilon \geq 1$, then the results from the previous section would imply Lipschitz viscosity solutions are stationary, provided one can show that the components solve uniformly elliptic equations of the form \eqref{uniellip}.
\end{rmk}

In the case $n = 2$, Lipschitz viscosity solutions of \eqref{uniellip} are $C^{1, \alpha}$ (see Theorem 12.4 in \cite{GT} and \cite{Je}). Applying the Allard Theorem, we find:

\begin{coro}\label{viscsmooth}
    Suppose $u: B_1^2(0) \rightarrow \mathbb{R}^m$ is a Lipschitz viscosity solution to the minimal surface system with components $u^\beta$ satisfying uniformly elliptic equations of the form \eqref{uniellip}. Then $u$ is real analytic.
\end{coro}

For a Lipschitz viscosity solution $u$ to the minimal surface system, it is reasonable to expect the components of $u$ to be viscosity solutions to uniformly elliptic scalar equations \eqref{uniellip} with ellipticity constants depending on the Lipschitz constant since this agrees with the case when $u$ is a $C^2$ classical solution of the system. However, it is not clear how to prove this is the case due to inadequate control on the derivative of $u$. Similar issues are encountered when working with small perturbations of the component functions of $u$ so a convergence argument does not seem fruitful. Nonetheless, these issues above can be avoided if we assume also that $\tilde{u} \in C^{1,\alpha}$ for some $\alpha \in (0,1]$, where $\tilde{u} := (u^2, \ldots, u^m)$. As a proof of concept, we sketch the argument in this case.

Set $M := \Lip u$ and suppose $\varphi \in C^2$ touches $u$ from above at $x_0 \in B_1^n(0)$ and satisfies \eqref{preg1} for $\lambda = 1$ and $\Lambda$ to be determined. After a translation of the graph of $u$, we may assume that $x_0 = 0$ and $u(0) = 0$. Furthermore, since the minimal surface system is invariant under the Lipschitz re-scaling $u(x) \mapsto \delta^{-1}u(\delta x)$ for any $\delta > 0$, we may choose $\delta$ small so that $\norm{D^2 \varphi}_\infty < (4n)^{-1}$. Let $\ell: \mathbb{R}^n \rightarrow \mathbb{R}^{m-1}$ be the linear part of $\tilde{u}$ at $x_0 = 0$. Localizing and applying a $C^{1,\alpha}$ Taylor expansion near the origin, we see that $\mathcal{G}_u$ touches $\{G = 0\}$ from the side $\{G \leq 0\}$ at the origin, where
$$
   G(x,z) := z^1 - \varphi - \frac{|x|^{2 + 2\alpha}}{\epsilon^2} + |\tilde{z} - \ell|^2
$$
and $\epsilon > 0$ is chosen depending on the $C^{1,\alpha}$ norm for $\tilde{u}$. Arguing as in Lemma 1.1 in \cite{Sa2}, it is not hard to show $G$ is a comparison function in a neighborhood of the origin in $\mathbb{R}^{n+m}$ for large enough $\Lambda$ depending on $n$ and $M$. 

To see this, set 
\begin{equation}
    G_1(x,z) := z^1 - \varphi \text{ and } G_2(x,z) := - \frac{|x|^{2 + 2\alpha}}{\epsilon^2} + |\tilde{z} - \ell|^2. \label{G2}
\end{equation}
and notice that
$$
    \nabla G(0,0) = \nabla G_1(0,0) = (-\nabla_{\mathbb{R}^n} \varphi(0), 1, 0, \ldots, 0) 
$$
while
$$
    D^2G_2(0,0) = D^2(|\tilde{z} - \ell|^2) \geq 0.
$$
Hence, at the origin $G$ and its first and second order derivatives behave like $H$ in Lemma 1.1 in \cite{Sa2}. It follows that there is a $\phi > 0$ small depending only on $M$ such that, for any $n$-dimensional subspace $L$ orthogonal to $\nabla G(0,0)$ lying outside the cone
$$
    \mathcal{C}_{\phi} := \Big\{(x,z) : |z| \leq \tan(\frac{\pi}{2} - \phi)|x|\Big\},
$$
we have
$$
    \Delta_L G(0,0) \geq -n \norm{D^2 \varphi}_\infty + \frac{1}{2} > 0.
$$
Otherwise, the projection $\pi_x$ of $L$ onto the $x$-subspace is non-singular with inverse bounded by a constant depending on $\phi$ (hence, on $M$). Then
$$
    \Delta_L G(0,0) \geq -\tr (\pi_x^T D^2 \varphi(0) \pi_x) \geq - n\mathcal{M}_{\Lambda}^{+}(D^2 \varphi) > 0.
$$
We thereby have proven:

\begin{prop}\label{alpha}
    If $u: B_1^n(0) \rightarrow \mathbb{R}^m$ is a Lipschitz viscosity solution to the minimal surface system with $\tilde{u} \in C^{1,\alpha}(B_1^n(0); \mathbb{R}^{m-1})$ for some $\alpha \in (0,1]$, then $u^1$ is a viscosity solution to some uniformly elliptic equation of the form \eqref{uniellip}, where $\lambda = 1$ and $\Lambda$ depends on $n$ and $\Lip u$.
\end{prop}

As an analogue to Corollary \ref{viscsmooth}, we have:

\begin{coro}\label{alphasmooth}
     If $u: B_1^2(0) \rightarrow \mathbb{R}^m$ is a Lipschitz viscosity solution to the minimal surface system with $\tilde{u} \in C^{1,\alpha}(B_1^2(0); \mathbb{R}^{m-1})$ for some $\alpha \in (0, 1]$, then $u$ is real analytic. In particular, when $n = m = 2$ then $C^{1,\alpha}$ regularity of a single component implies full regularity.
\end{coro}

Interestingly, one can also prove an $\epsilon$-regularity theorem for the case $u$ is Lipschitz and $\tilde{u}$ has small $C^{1,\alpha}$ norm, $\alpha \in (0,1]$. The proof is similar in spirit to that of Proposition 2.5 in \cite{Sa2} (i.e. Improvement of Flatness). Suppose we have a sequence $\{u_k\}$ of Lipschitz viscosity solutions to the system such that
    \begin{itemize}
        \item[(i)] $\tilde{u}_k \in C^{1,\alpha}$ in $B_1^n(0)$,
        \item[(ii)] $\norm{\tilde{u}_k}_{C^{1,\alpha}} \leq \epsilon_k$,
        \item[(iii)] $\epsilon_k \rightarrow 0$ as $k \rightarrow \infty$.
    \end{itemize}
    Points $(i)$ and $(ii)$ above guarantee that there is a subsequence (which we do not relabel) of the functions $w_k: B_{2^{-1}}^n(0) \rightarrow \mathbb{R} \times B_1^{m-1}(0)$ defined by 
    $$
        w_k := (u_k^1, \epsilon_k^{-1}\tilde{u}_k)
    $$
    which converge uniformly in $B_{2^{-1}}^n(0)$ to a Lipschitz function $w$ with $\Lip w \leq \max\{M, 1\}$, $\tilde{w} \in C^{1,\alpha}(B_{2^{-1}}^n(0); \mathbb{R}^{m-1})$, and $\norm{\tilde{w}}_{C^{1,\alpha}} \leq 1$.
    
\begin{lem}\label{MSEvisc}
    The first component $w^1$ satisfies the minimal surface equation in the viscosity sense in $B_{2^{-1}}^n(0)$.
\end{lem}

For simplicity of notation, we define the \emph{minimal surface operator} on $v \in C^2(\Omega)$ by
$$
    \mathcal{M}(\nabla v) := \diverge_{\mathbb{R}^n}\Bigg(\frac{\nabla v}{\sqrt{1 + |\nabla v|^2}}\Bigg).
$$
Hence, $v \in C^2(\Omega)$ satisfies the minimal surface equation in $\Omega$ if and only if $\mathcal{M}(\nabla v) = 0$ in $\Omega$. 

\begin{proof}
    The proof is by contradiction. Without loss of generality, assume there is an $x_0 \in B_{2^{-1}}^n(0)$ such that $w^1$ is touched from above by $\varphi \in C^2$ with $\mathcal{M}(\nabla_{\mathbb{R}^n} \varphi)(x_0) < 0$. By translation invariance, we may assume $x_0 = 0$. In addition, by continuity, for small $\eta > 0$ there is an $r_0 \in (0, \frac{1}{2})$ such that $\mathcal{M} (\nabla_{\mathbb{R}^n}(\varphi + \eta|x|^2)) < 0$ on $\overline{B_{r_0}^n(0)}$. Hence, we may assume quadratic separation at the origin. That is,
    \begin{equation}
        \varphi - w^1 \geq \eta |x|^2 \text{ for some } \eta > 0\text{ small.} \label{eta}
    \end{equation}
    
    By the uniform convergence of the $u_k^1$ to $w^1$, there is a translation of $\varphi$ we denote $\psi$ touching a $u_k^1$ from above at a point $y_0 \in B_{r_0}^n(0)$ for abitraily large $k$. Let $\ell$ be the linear part of $\tilde{w}_k$ at $y_0$. Define 
    $$
        G(x,z) = z^1 - \psi - |x - y_0|^{2\alpha + 2} + \epsilon_k^{-2}|\tilde{z} - \epsilon_k \ell|^2,
    $$
    set $Y_0 := (y_0, u_k(y_0))$, and observe that $G(x,z) \leq 0$ near $Y_0$ with $G(Y_0) = 0$. We have
    $$
        \nabla G(Y_0)= (- \nabla_{\mathbb{R}^n} \psi(y_0), 1, 0, \ldots, 0) = \sqrt{1 + |\nabla_{\mathbb{R}^n} \psi(y_0)|^2} N(y_0),
    $$
  where $N$ is the outward-pointing unit normal to the graph of $\psi$ embedded in $\mathbb{R}^{n+m}$. Thus, if $L$ is any $n$-dimensional subspace lying in the $(x,z^1)$-subspace orthogonal to $\nabla G(Y_0)$, $L$ must be the tangent space to the graph of $\psi$ at $y_0$. Set 
   $$
        G_1(x,z) := z^1 - \psi \text{ and } G_2(x,z) := - |x - y_0|^{2\alpha + 2} + \epsilon_k^{-2}|\tilde{z} - \epsilon_k \ell|^2
   $$
   and observe that
   $$
        \nabla G_2(Y_0) = 0 \text{ and } D^2 G_2(Y_0) \geq 0.
   $$
   We then have
   $$
        \Delta_L G(Y_0) \geq \Delta_L G_1(Y_0) = -n^{-1}\sqrt{1 + |\nabla_{\mathbb{R}^n} \psi|^2} \mathcal{M}(\nabla_{\mathbb{R}^n} \psi)(y_0) > 0.
   $$
   Using continuity, this inequality holds at $Y_0$ for all $n$-planes $L$ orthogonal to $\nabla G(Y_0)$ lying in the cone
   $$
        \mathcal{C}_\phi := \{(x,z) \in \mathbb{R}^{n+m}: |\tilde{z}| \leq \tan \phi |(x,z^1)| \}.
   $$
   Furthermore, since $D^2 G_2(Y_0) \geq 0$, $\phi$ can be chosen depending on $\Lip w$ independent of $\epsilon_k$. Suppose now that $\xi \in \mathbb{R}^{n+m}$ is a unit vector with $\xi \notin \mathcal{C}_\phi$. Then $|\tilde{\xi}| \geq \gamma$ for some $\gamma > 0$ depending on $\phi$ (hence, on $M$). Using that $|D\ell(y_0)| \leq 1$, we find that for any $n$-plane $L$ orthogonal to $\nabla H(Y_0)$ lying outside $\mathcal{C}_\phi$ 
   $$
        \Delta_L G(Y_0) = \Delta_L G_1(Y_0) + \Delta_L G_2(Y_0) \geq -n \norm{D^2 \psi} + \frac{\gamma^2}{\epsilon_k^2} - \frac{C}{\epsilon_k} > 0
   $$
   where $C$ is a constant independent of $\epsilon_k$. It follows that $\Delta_L G(Y_0) > 0$ for any $n$-dimensional subspace $L$ orthogonal to $\nabla G(Y_0)$. By continuity, this holds in a neighborhood of $Y_0$. It follows that $G$ is a comparison function for some $u_k$ for large enough $k$ near $y_0$, which is a contradiction. 
\end{proof}

Solving the Dirichlet problem for the minimal surface equation in $B_{2^{-1}}^n(0)$ with boundary data $w^1$, we see that $w^1$ agrees with the the unique solution to the Dirichlet problem so $w^1 \in C^\infty(B_{2^{-1}}^n(0)) \cap C^{0,1}(\overline{B_{2^{-1}}^n(0)})$ (see Chapter 13 in \cite{G}). The $\epsilon$-regularity theorem now follows from derivative estimates for the limit function $w^1$ in Lemma \ref{MSEvisc} and Savin's Allard Theorem.

\begin{prop}[$\epsilon$-Regularity Theorem]\label{allardtype}
    Suppose $u: B_1^n(0) \rightarrow \mathbb{R}^m$ is a Lipschitz viscosity solution to the minimal surface system with $\tilde{u} \in C^{1,\alpha}(B_1^n(0); \mathbb{R}^{m-1})$ for some $\alpha \in (0,1]$. Then there are small universal constants $r_0,\epsilon_0 > 0$ such that $u$ is real analytic in $B_{r_0}^n(0)$ provided $\norm{\tilde{u}}_{C^{1,\alpha}} \leq \epsilon_0 \text{ in } B_1^n(0)$.
\end{prop}

Roughly speaking, Proposition \ref{allardtype} says that Lipschitz viscosity solutions to the system which lie close to embedded codimension one graphs of $\mathbb{R}^{n+m}$, in the sense that they are $n$-dimensional graphs in $\mathbb{R}^{n+m}$ that can be embedded as a codimension one graph in some Euclidean space, are analytic. It would be interesting to replace the condition that $\tilde{u}$ has small $C^{1,\alpha}$ norm with a looser condition, such as $\Lip \tilde{u}$ is small. 

With the discussion above in mind, we are left with the following questions:
\begin{ques}\label{ques1}
   Can the condition $\tilde{u} \in C^{1,\alpha}$ in Proposition \ref{alpha} and Proposition \ref{allardtype} be loosened to require, for example, only that $u$ is Lipschitz in the former and $\Lip \tilde{u}$ is small in the latter? 
\end{ques}

We conclude this section by noting that if the answer to Question \ref{ques1} is positive, then by Corollary \ref{viscsmooth} all Lipschitz viscosity solutions to the system are real analytic when $n= 2$. Notice also that if the answer to Question \ref{ques1} is positive, then we immediately get a rotationally invariant strong maximum principle\footnote{That is, if the graph of a viscosity solution $u:B_1^n(0) \rightarrow \mathbb{R}^m$ touches a hyperplane $L \subset \mathbb{R}^{n+m}$ from one side it lies entirely in $L$.} for Lipschitz viscosity solutions to the system due to the strong maximum principle for scalar equations and the rotational invariance of the system \eqref{MSS}. Finally, we point out that Proposition \ref{alpha} suggests it may be helpful to develop a suitable regularization for viscosity solutions to the system.

\section{Stationary Lipschitz Weak Solutions} 
We now make rigorous the notion that Lipschitz weak solutions to the system \eqref{MSS} which are invariant under small coordinate rotations are stationary. The key idea is that rotation invariance allows one to collect a basis of directions for $\mathbb{R}^{n+m}$ in which the graph of $u$ is stationary along each basis direction. Theorem \ref{wstatnry} below serves as a partial resolution to Conjecture 2.1 in \cite{LO}. We will make use of the following definitions: 

\begin{defn}
Suppose $\Gamma$ is a Lipschitz submanifold of $\mathbb{R}^{n+m}$.
\begin{enumerate}
    \item Let $v \in \mathbb{R}^{n+m}$. We say that $\Gamma$ is stationary with respect to $v$ if $\delta\Gamma(X) = 0$ whenever 
    $$
        X(x,z) = \varphi(x,z)v
    $$
    for some $\varphi \in C_0^1(\mathbb{R}^{n+m})$ that vanishes in a neighborhood of $\partial \Gamma$.
    \item Let $(x,z) := (x^1,\ldots, x^n,z^1,\ldots, z^m)$ be coordinates on $\mathbb{R}^{n+m}$. We say a vector field $X \in C_0^1(\mathbb{R}^{n+m}; \mathbb{R}^{n+m})$ is coordinate vertical in $\mathbb{R}^{n+m}$ if 
    $$
        X(x,z) = (0, \varphi(x,z)) 
    $$
    for some $\varphi \in C_0^1(\mathbb{R}^{n+m}; \mathbb{R}^{m})$.
    \item We say that a variation $\varphi_t$ of $\Gamma$ is coordinate vertical in $\mathbb{R}^{n+m}$ if
    $$
        \pdv{\varphi_t(x,z)}{t} \Big \lvert_{t=0} = X(x,z)
    $$
    for a vector field $X \in C_0^1(\mathbb{R}^{n+m}; \mathbb{R}^{n+m})$ vanishing in a neigborhood of $\partial \Gamma$ which is coordinate vertical in $\mathbb{R}^{n+m}$.
\end{enumerate}
\end{defn}

Taking $\varphi_t$ to be a coordinate vertical variation of $\mathcal{G}_u$ in $\mathbb{R}^{n+m}$, the first variation formula and approximation by smooth functions with compact support yields the lemma below. Since the argument is standard, we leave the details to the reader.
\begin{lem}\label{vert}
    If $u: \Omega \rightarrow \mathbb{R}^m$ is a Lipschitz weak solution to the minimal surface system, then $\delta \mathcal{G}_u(X) = 0$ whenever $X \in C_0^1(\mathbb{R}^{n+m}; \mathbb{R}^{n+m})$ is coordinate vertical and vanishes in a neighborhood of $\partial \mathcal{G}_u$.
\end{lem}

Let $e_1, \ldots, e_{n+m}$ be a basis for $\mathbb{R}^{n+m}$. Since any vector field $X \in C_0^1(\mathbb{R}^{n+m}; \mathbb{R}^{n+m})$ can be written as a linear combination $X(x,z) = a^i(x,z)e_i$ with $a^i \in C_0^1(\mathbb{R}^{n+m})$ for each $i = 1, \ldots, n+m$, the next lemma is an immediate consequence of the first variation formula and linearity of the integral.

\begin{lem}\label{indep}
    Let $\Gamma$ be a Lipschitz submanifold of $\mathbb{R}^{n+m}$. If there is a basis $e_1, \ldots, e_{n+m}$ for $\mathbb{R}^{n+m}$ for which $\Gamma$ is stationary with respect to each $e_i$, then $\Gamma$ is stationary.
\end{lem}

In the theorem below, we define the cone $\mathcal{C}_\phi$ by
\begin{equation}
    \mathcal{C}_\phi := \{(x,z) : |x| \leq \tan \phi |z|\}. \label{cone}
\end{equation}
Using the previous two lemmata, we have:
\begin{thm}\label{wstatnry}
    Let $u: \Omega \rightarrow \mathbb{R}^m$ be a Lipschitz weak solution to the minimal surface system and let $f_1,\ldots, f_m$ be a basis for $\mathbb{R}^m$ in $\mathbb{R}^{n+m}$. Suppose that there is a $\phi$ such that the image of the graph of $u$ remains a graphical Lipschitz weak solution over its projection onto $\Omega$ under any orthogonal transformation $U$ of $\mathbb{R}^{n+m}$ such that $Uf_1 \in \mathcal{C}_{\phi} \setminus \Span \{f_2, \ldots, f_m\}$, where $\mathcal{C}_{\phi}$ is as in \eqref{cone}. Then $u$ is stationary.
\end{thm}
\begin{proof}
   If $\phi$ is small and $U$ is any orthogonal transformation of $\mathbb{R}^{n+m}$ such that $Uf_1 \in \mathcal{C}_{\phi} \setminus \Span \{f_2,\ldots, f_m\}$, then $U(\mathcal{G}_u)$ remains graphical over a (possibly smaller) domain $\overline{\Omega} \subset \Omega$; namely, the projection of $U(\mathcal{G}_u)$ onto $\Omega$. Choose coordinates so that $U(\mathcal{G}_u)$ is the graph of $\overline{u} : \overline{\Omega} \rightarrow \mathbb{R}^m$ and let $\overline{\Sigma}$ be the singular set for $\overline{u}$ in $\overline{\Omega}$. Then $\delta U(\mathcal{G}_u)(X) = 0$ whenever $X$ is coordinate vertical and vanishes in a neighborhood of $\partial U(\mathcal{G}_u)$ by Lemma \ref{vert}. Let $\overline{\varphi}_t$ be a coordinate vertical variation of $U(\mathcal{G}_u)$ with
    $$
        \overline{X}(x,z) := \pdv{\overline{\varphi}_t(x,z)}{t} \Big\lvert_{t = 0} = \overline{a}(x,z)f_1
    $$
    where $\overline{a} \in C_0^1(\mathbb{R}^{n+m})$ vanishes near $\partial U(\mathcal{G}_u)$. Then $\varphi_t := U^{-1} \circ \overline{\varphi}_t \circ U$ is a variation of $\mathcal{G}_u$ and
    $$
        X(x,z) := \pdv{\varphi_t(x,z)}{t}\Big \lvert_{t = 0} = a(x,z)U^{-1}f_1
    $$
    where $a^j(x,z) = \overline{a}^j \circ U \in C_0^1(\mathbb{R}^{n+m})$ vanishes near $\partial \mathcal{G}_u$. By the first variation formula and \eqref{delta},
    $$
        \delta\mathcal{G}_u(X) = \delta U(\mathcal{G}_u)(\overline{X}) = 0.
    $$
    It follows that $\mathcal{G}_u$ is stationary with respect to $U^{-1}f_1$. Set $U_1 = U$ and suppose $U$ was chosen so that $\{U^{-1}f_1, f_1,\ldots,f_m\}$ form a linearly independent set. Repeating this process inductively, we obtain a basis $\mathcal{B} := \{(U_i^{-1}f_1)_{i = 1}^n, (f_j)_{j = 1}^m\}$ for $\mathbb{R}^{n+m}$ for which $\mathcal{G}_u$ is stationary with respect to each vector in $\mathcal{B}$. Applying Lemma \ref{indep}, we conclude that $\mathcal{G}_u$ is stationary.
\end{proof}

To see how one might apply Theorem \ref{wstatnry}, we briefly sketch a second proof of Corollary \ref{ViscStat}, which is the same in principle. If $u: \Omega \rightarrow \mathbb{R}^m$ is a Lipschitz viscosity solution to the system with singular set $\Sigma$ satisfying $\mathcal{H}^{n-1}(\Sigma) = 0$, we may integrate by parts in the expression \eqref{weak} with $X$ a coordinate vertical vector field away from the singular set to conclude that $u$ is a weak solution to \eqref{MSS}. Since Lipschitz viscosity solutions are rotationally invariant and small rotations of their graphs remain graphical Lipschitz viscosity solutions with $\mathcal{H}^{n-1}$-measure zero singular set, small rotations of the graphs of Lipschitz viscosity solutions with $\mathcal{H}^{n-1}$-measure zero singular sets are the graphs of Lipschitz weak solutions; hence, are stationary by Theorem \ref{wstatnry}. More generally, to prove that any rotationally invariant class of solutions to \eqref{MSS} is stationary, it is enough to show they are weak solutions to the system.

\section{Interior Gradient Estimate}

In \cite{MW}, M.T. Wang proved an interior gradient estimate for smooth solutions $u$ to the minimal surface system using integral methods under the additional assumption that $u$ satisfies the \emph{area-decreasing condition}. 

\begin{defn}[Area-Decreasing Condition] \label{AD}
    Suppose $u: \Omega \rightarrow \mathbb{R}^m$ is a smooth solution to the minimal surface system. Then $u$ satisfies the area-decreasing condition provided the Jacobian of $Du: \mathbb{R}^n \rightarrow \mathbb{R}^m$ is less than or equal to $1$ on any two-dimensional subspace of $\mathbb{R}^n$. If $\lambda_i$ denote the singular values for $Du$, then the area-decreasing condition is equivalent to $|\lambda_i \lambda_j| \leq 1$ for $i \neq j$.
\end{defn}

The underlying principle of the proof is that the area-decreasing condition defines a region on the Grassmannian of $n$-planes in $\mathbb{R}^{n+m}$ on which $v := \sqrt{\det g}$ is a convex function, where $g$ represents the usual Riemannian metric \eqref{metric} on the graph of $u$. Notice that this differs from the codimension one case in which $v$ is always convex. Here, we show that an interior gradient estimate can be obtained using the maximum principle provided the $L^\infty$ norm of every component but one is small. The proof is based on a method by Korevaar in \cite{Ko} for constant mean curvature equations and relies on the fact that $v^{-\frac{1}{n}}$ is superharmonic. This is expressed in the following lemma:

\begin{lem}
    Suppose $u: B_1^n(0) \rightarrow \mathbb{R}^m$ is a smooth solution to the minimal surface system satisfying the area-decreasing condition. Set $w := \ln v$. Then
    $$
        \Delta_g w \geq \frac{1}{n}\norm{\nabla_g w}_g^2,
    $$
    where $\Delta_g$ and $\nabla_g$ are the Laplace-Beltrami operator and the Levi-Civita connection on $\mathcal{G}_u$, respectively.
\end{lem}

For the proof of the lemma, see Lemma 2.1 in \cite{MW}. 

\begin{prop}[Interior Gradient Estimate]\label{gradest}
    Suppose $u: B_1^n(0)  \rightarrow \mathbb{R}^m$ is a smooth solution to the minimal surface system satisfying the area-decreasing condition. Suppose further that $u^1 \geq 0$ and $\norm{u^\beta}_\infty < \epsilon$ for each $\beta = 2,\ldots, m$. Then there is an $\epsilon_0 > 0$ depending only on $m,$ $n$, and $\norm{u}_{\infty}$ such that 
$$
    |Du(0)| \leq C_1e^{C_2\norm{u}_\infty^2} \text{ provided } \epsilon \leq \epsilon_0,
$$
where $C_1$ depends on $n$, $m$, and $\norm{u}_\infty$ and $C_2$ depends on $n$ and $m$.
\end{prop}
\begin{proof}
    First, we construct a smooth test function $h$ satisfying 
    \begin{itemize}
        \item[(i)] $h(0,u(0)) > 0$;
        \item[(ii)] $h(x,z) < 0$ on $|x| = 1$;
        \item[(iii)] $\Delta_L h > 0$ along all $n$-planes in $\mathbb{R}^{n+m}$ containing at least one $z$-direction\footnote{As before, $\Delta_L h$ denotes the Laplacian of $h$ restricted to the $n$-plane $L$.} .
    \end{itemize}
    After a rigid motion of the graph of $u$, we may assume without loss of generality that $u(0)$ lies on the positive $z^1$-axis. Set $M := 2\norm{u}_{L^\infty(B_1(0))}$. If $M = \infty$ or $M = 0$ the result is trivial. Thus, we may assume $0 < M < \infty$. Define $\psi: \mathbb{R}^{n+m} \rightarrow \mathbb{R}$ by
$$
    \psi(x,z) = z^1 + 2M|x|^2 - A|\tilde{z}|^2
$$
where $\tilde{z} := (z^2, \ldots, z^m)$ and $A$ is to be determined. Set 
$$
    h(x,z) = e^{-B\psi(x,z)} - e^{-BM}
$$
for $B$ to be determined. Computing, we see that for any unit vector $\xi = (\xi_x, \xi_z) \in \mathbb{R}^{n} \times \mathbb{R}^m$
\begin{align*}
    D_{\xi \xi} h(x,z) &= 16M^2B^2e^{-B\psi}\sum_{i,j =1}^n x^ix^j\xi_x^i\xi_x^j - 4MBe^{-B\psi}|\xi_x|^2 + 8MB^2e^{-B\psi} \sum_{j = 1}^n x^j \xi_x^j \xi_z^1 \\
    &- 16MAB^2 e^{-B\psi}\sum_{i = 1}^n\sum_{j=2}^m x^iz^j\xi_x^i\xi_z^j +B^2e^{-B\psi}|\xi_z^1|^2 - 4AB^2e^{-B\psi}\sum_{j = 2}^m z^j\xi_z^j\xi_z^1 \\
    &+ 4A^2B^2e^{-B\psi}\sum_{i=2}^m\sum_{j=2}^mz^iz^j\xi_z^i\xi_z^j + 2ABe^{-B\psi}|\xi_{\tilde{z}}|^2.
\end{align*}
We henceforth assume $A > B$. If $\xi = (0,\ldots,0,\xi^1,\ldots, \xi^m)$ is a unit vector lying in the $z$-subspace, then 
\begin{align*}
    D_{\xi \xi} h(x,z) 
    &\geq B^2e^{-B\psi}|\xi^1|^2 + ABe^{-B\psi}|\xi_{\tilde{z}}|^2 - 4AB^2m\epsilon e^{-B\psi} + A^2B^2e^{-B\psi}\Big(\sum_{j = 2}^m z^j\xi^j \Big)^2 \\
    &\geq B^2e^{-B\psi} - 4AB^2m\epsilon e^{-B\psi} \\
    &>0
\end{align*}
provided $\epsilon < (8Am)^{-1}$. Thus, if $\epsilon$ is chosen small, the second order derivatives of $h$ in the $z$-direction are positive. We show that we can choose constants $A$ and $B$ depending only on $m$, $n$, and $M$ so that the second order derivatives of $h$ in the $z$-directions overpower all second order derivatives of $h$ in orthogonal directions. Let $\xi = (0,\ldots, 0, \xi^1, \ldots, \xi^m)$ and suppose $\tilde{\xi} = (\tilde{\xi}_x, \tilde{\xi}_z)$ is a unit vector orthogonal to $\xi$. Then 
$$
    \tilde{\xi}_z^1 \xi^1 + \tilde{\xi}_{\tilde{z}} \cdot \xi_{\tilde{z}} = 0.
$$
Using that $|\xi| = 1$, we see that 
$$
    0 = \lim_{|\xi_{\tilde{z}}| \rightarrow 0} (\tilde{\xi}_z^1 \xi^1 + \tilde{\xi}_{\tilde{z}} \cdot \xi_{\tilde{z}}) = \tilde{\xi}_{z}^1.
$$
In particular, given $\tilde{\gamma} > 0$ there is a $\gamma  > 0$ so that $|\xi_{\tilde{z}}| \leq \gamma$ implies $|\tilde{\xi}_z^1| \leq \tilde{\gamma}$. Without loss of generality, we may assume $\gamma = \tilde{\gamma}$. We have two cases to consider: $|\xi_{\tilde{z}}| \geq \gamma$ and $|\xi_{\tilde{z}}| \leq \gamma$.

If $|\xi_{\tilde{z}}| \leq \gamma$, then $|\tilde{\xi}_z^1| \leq \gamma$ also. Using the general expression for the second-order derivatives above, we find
$$
    D_{\tilde{\xi} \tilde{\xi}} h(x,z)\geq -4MBe^{-B\psi} - 8MB^2n \gamma e^{-B\psi} - 16 AB^2 Mmn\epsilon e^{-B\psi} - 4AB^2 m\gamma \epsilon e^{-B\psi}.
$$
Then there is a $\overline{C}_1 > 0$ large and independent of $A$, $M$, $m$, and $n$ such that if
$$
    \epsilon < \min\{(\overline{C}_1AMmn^2)^{-1}, (\overline{C}_1Amn^2)^{-1}\}
$$ 
and 
$$
    \gamma < \min\{(\overline{C}_1Mn^2)^{-1}, (\overline{C}_1n^2)^{-1}\}
$$
we have
$$
    2nD_{\tilde{\xi} \tilde{\xi}} h(x,z) \geq -8nMB e^{-B\psi} - \frac{B^2}{4} e^{-B\psi}.
$$
On the other hand, having chosen $\epsilon$ to satisfy the condition above, we find
$$
    D_{\xi \xi} h(x,z) \geq \frac{B^2}{2}.
$$
Thus, choosing $B > \overline{C}_2Mn$ for $\overline{C}_2$ large independent of $M$ and $n$ while adjusting $A$ accordingly, we have
$$
    D_{\xi \xi} h(x,z) + 2n D_{\tilde{\xi} \tilde{\xi}} h(x,z) > 0.
$$
Suppose now that $|\xi_{\tilde{z}}|\geq \gamma$. If $\epsilon$ satisfies the properties above, then
$$
    D_{\xi \xi} h(x,z) \geq 2ABe^{-B\psi}\gamma^2 - \frac{B^2}{2}e^{-B\psi}
$$
and
$$
    2nD_{\tilde{\xi} \tilde{\xi}} h(x,z) \geq -8nMBe^{-B\psi} - \frac{B^2 e^{-B\psi}}{3} 
$$
If $A$ is chosen so that 
$$
    A > \max\{\overline{C}_3M^3n^5, \overline{C}_3n^5\}
$$ 
for a constant $\overline{C}_3$ large depending on the choices for $\overline{C}_1$ and $\overline{C}_2$,
then $A > B$ and adding the inequalities above and rearranging terms shows
$$
    D_{\xi \xi} h(x,z) + 2n D_{\tilde{\xi} \tilde{\xi}} h(x,z) > 0.
$$
Combining the previous two estimates shows that $\Delta_L h > 0$ along any $n$-dimensional subspace $L$ containing a $z$-direction. Furthermore, 
$$
    h(0,u(0)) = e^{-Bu^1(0)} - e^{-BM} > 0
$$ 
while on $|x| = 1$
$$
    h(x,z) \leq e^{-B(2M - A\epsilon^2)} - e^{-BM} < 0
$$
provided $\epsilon < \sqrt{\frac{M}{A}}$. Hence, if $\epsilon < \min\{ M^{\frac{1}{2}}A^{-\frac{1}{2}}, (\overline{C}_1AMmn^2)^{-1}\}$, $h$ satisfies properties (i)-(iii).

For any $\phi \in (0, \frac{\pi}{2})$, let $\mathcal{C}_{\phi}$ be the cone 
$$
   \mathcal{C}_{\phi} := \{(x,z): |z| \leq \tan \phi|x|\}
$$
and note that by smoothness of $h$ and property (iii) above, there is a small constant $\phi := \phi(n,m, \norm{u}_{\infty})$ such that $\Delta_L h > 0$ for any $n$-dimensional subspace $L$ lying outside of $\mathcal{C}_{\frac{\pi}{2} - \phi}$. Define $v: \mathbb{R}^n \rightarrow \mathbb{R}$ by 
$$
    v(x) := \sqrt{\det g(x)} = \sqrt{\det(I + Du^T(x)Du(x))}.
$$
Since $u$ satisfies the area-decreasing condition, $v^{-\frac{1}{n}}$ is superharmonic. Let $\tilde{h} := h(x, u(x)))$ be the restriction of $h$ to the graph of $u$. Define $U_+ := \{ x \in B_1^n(0) : \tilde{h} > 0 \}$ and note that $U_+ \subset \subset B_1^n(0)$ by properties (i) and (ii) above. In particular, $v^{-\frac{1}{n}}$ is bounded below on $\overline{U_+}$. It follows that there is a $t_0 > 0$ and $x_0 \in \overline{U_+}$ such that $t_0\tilde{h}$ touches $v^{-\frac{1}{n}}$ from below at $x_0$. Note that the smallness of $t_0$ correlates with the largeness of $|Du|$ at the point of contact. Hence, there is a $T > 0$ depending on $m$, $n$, and $\norm{u}_{\infty}$ such that, if $t_0 \leq T$, the tangent space to the graph of $u$ at $x_0$ lies outside of $\mathcal{C}_{\frac{\pi}{2} - \phi}$. However, in this case the superharmonic function $v^{-\frac{1}{n}}$ is touched from below by the smooth subharmonic function $t_0\tilde{h}$ at $x_0$ which cannot be. It follows that $t_0 > T$. In particular, 
$$
v^{-\frac{1}{n}}(x) \geq T\tilde{h}(x) \text{ in } B_1(0).  
$$
Rearranging terms in the preceding inequality gives
$$
    \frac{e^{C\norm{u}_\infty^2}}{T} \geq \sup_{x \in B_1(0)}v(x) \geq |Du(0)|
$$
where $C$ depends on $n$ and $m$ and $T$ depends on $n$,$m$, and $\norm{u}_{L^\infty(B_1(0))}$.
\end{proof}

We briefly remark that Proposition \ref{alpha} and its corollaries along with Proposition \ref{gradest} suggest that maximum principle techniques may remain viable for solutions to the minimal surface system when, heuristically, their graphs are small perturbations of embedded codimension one submanifolds.

\section*{Acknowledgements}
I gratefully acknowledge the support as a graduate student researcher through C. Mooney's Sloan Fellowship and UC Irvine Chancellor's Fellowship. I would also like to thank my thesis advisers C. Mooney and R. Schoen for their patient guidance and feedback on the drafts of this paper.

\end{document}